\documentclass{amsart}
\usepackage{amsthm,amstext,amsmath,amscd,amssymb,latexsym}
\usepackage{color}
\usepackage[matrix,arrow]{xy}
\usepackage{amsfonts,enumerate,array}
\usepackage[colorlinks=true]{hyperref}
\usepackage{todonotes}
	\usepackage{verbatim}
\usepackage[toc,page]{appendix}

\newcommand{\PP}{{\mathbb P}}

\newcommand{\LL}{{\mathcal L}}

\newcommand{\II}{{\mathcal I}}

\newcommand{\Bs}{{\rm{Bs}}}

\newcommand{\ls}{{\mathcal{L}}}

\DeclareMathOperator{\Pic}{Pic}
\DeclareMathOperator{\Cr}{Cr}

\DeclareMathOperator{\J}{J}

\DeclareMathOperator{\sldim}{\sigma ldim}
\DeclareMathOperator{\Exc}{Exc}

\newcommand{\paper}{: \begin{it}}
\newcommand{\jour }{, \end{it}}

\newtheorem{theorem}{Theorem}[section]
\newtheorem{lemma}[theorem]{Lemma}
\newtheorem{proposition}[theorem]{Proposition}
\newtheorem{corollary}[theorem]{Corollary}
\newtheorem{conjecture}[theorem]{Conjecture}

\newtheorem{question}[theorem]{Question}
\theoremstyle{definition}
\newtheorem{definition}[theorem]{Definition}
\newtheorem{notation}[theorem]{Notation}
\newtheorem{remark}[theorem]{Remark}

\newtheorem{example}[theorem]{Example}

\theoremstyle{remark}

\numberwithin{equation}{section}

\definecolor{mygreen}{rgb}{0.13, 0.55, 0.13}

\title{Positivity of divisors on blown-up projective spaces, II}

\author{Olivia Dumitrescu}
\address{
Olivia Dumitrescu:
Central Michigan University\\
Pearce Hall 209\\
Mt. Pleasant, Michigan 48859, US\\}

\address{and Simion Stoilow Institute of Mathematics\\
Romanian Academy\\
21 Calea Grivitei Street\\
010702 Bucharest, Romania}
\email{dumit1om@cmich.edu}

\thanks{The first author is supported by the Simons Foundation travel grant 524545}

\author{Elisa Postinghel}
\address{Elisa Postinghel:
Department of Mathematical Sciences, Loughborough University, LE11 3TU, UK}
\email{E.Postinghel@lboro.ac.uk}



\begin{document}

\begin{abstract} 
We construct log resolutions of pairs on the blow-up of the projective space in an arbitrary number of general points and we discuss the semi-ampleness of the strict transforms. 
As an application we give an explicit proof that the abundance conjecture holds for an infinite family of such pairs.

For $n+2$ points, these strict transforms are F-nef divisors on the moduli space $\overline{\mathcal{M}}_{0,n+3}$
in a Kapranov's model: we show that all of them are nef.
\end{abstract}

\maketitle
{ \hypersetup{linkcolor=black} \tableofcontents}
\setcounter{section}{0}

\section*{Introduction}

 This paper studies positivity questions for divisors on   blow-ups of projective spaces of arbitrary dimension in points in  general position and along the linear cycles these points span.

 Let $X_{s,(0)}$ denote the blow-up of $\PP^n$ at $s$ points in general position.
An interesting problem is the \emph{dimensionality problem} that consists in the computation of dimension of the space of global sections of the sheaves associated to effective divisors on $X_{s,(0)}$. This is related to the \emph{polynomial interpolation problem} for homogeneous polynomials of fixed degree and multiplicity at a collection of $s$ general points.
This problem is longstanding and is related to the Nagata's conjecture and the Segre-Harbourne-Gimigliano-Hirschowitz conjecture for $n=2$, see \cite{Ciliberto} for an account, and to Fr\"oberg-Iarrobino conjectures for ideals generated by forms. 
 In \cite{bdp1} the authors solved this problem for divisors on $X_{s,(0)}$
for  $s\le n+2$ and in  \cite{bdp3} they gave a conjectural formula for the dimension 
of linear systems on  $X_{n+3,(0)}$ that takes into account the contributions given 
by the rational normal curve and the joins of its secants with linear subspaces.
In \cite{dp-pos1} further positivity properties such as basepoint freeness and, more generally, $l$-very ampleness of linear series on $X_{s,(0)}$, for arbitrary $s$, are studied: these have "large" degree with respect to the multiplicities of the points.

This manuscript analyses all other linear series, i.e. those with low degree whose general member has singularities. In particular we construct resolutions of these singularities and we prove that proper transform of divisors are basepoint free.

 Let $X_{s,(n-2)}$ denote the subsequent blow-up of $X_{s,(0)}$  along all the linear subspaces of $\PP^n$ spanned by the $s$ points in increasing dimension. Of particular interest are the cases $s=n+1$ and $s=n+2$ that correspond to the compactifications of the \emph{Losev-Manin moduli space} (see \cite{LosevManin}) and of the {\it moduli space of stable rational marked curves} $\overline{\mathcal{M}}_{0,n+3}$ (see \cite{Ka}) respectively.
In this article we pose a series of questions about vanishing cohomology of strict transforms on $X_{s,(n-2)}$  of
divisors on $X_{s,(0)}$ of which a positive answer would
imply a solution to the dimensionality problem, see Subsection \ref{section vanishing cohomology n+3}. 

We recall that a divisor $D$ on $X_{s,(0)}$ is said to be \emph{only linearly obstructed} if the dimension of the corresponding linear system equals its \emph{linear expected dimension} introduced in \cite{bdp1}. 
In \cite{bdp1,dp} it was shown that all effective as well as certain non effective divisors $D$, that are only linearly obstructed, can be birationally modified,  by blowing-up their linear base locus and subtracting the linear divisorial components; this produced divisors $\tilde{D}$, the strict transforms of $D$, whose  higher cohomology groups  vanish. An application of these vanishing theorems that will be developed in this paper is the description of {\it globally generated} and  {\it semi-ample} divisors (Theorem \ref{gg theorem Xr}).
Furthermore  we extend this result to a class of divisors on $X_{n+3,(0)}$ that can have non-linear base locus, by taking the blow-up along the joins between linear cycles and secant varieties to the unique rational normal curve of degree $n$ through the $s$ points (Theorem \ref{base point free tilde-sigma-D}).

The varieties $X_{n+3,(0)}$ were very much studied recently as a source of interesting and explicit examples of Fano type manifolds in higher dimension. 
If  $s\le n+3$,  $X_{s,(0)}$  is log Fano, hence in particular a \emph{Mori dream space} (see for example \cite{am,CT}). 
The space $X_{n+3,(0)}$ is of particular interest because it is  the
 {\it moduli space of parabolic vector bundles}  of rank $2$ over $\PP^1$ (see \cite{bau}, \cite{Mu,Moon}). 
The Cox ring $X_{n+3,(0)}$, end therefore extremal rays of the effective cone of divisors, was given in \cite{CT}.
Equations for the facets of the effective cone and the movable cone of divisors
 were computed by the two authors of this manuscript together with Brambilla in  \cite{bdp3}.
These were found based on a \emph{base locus lemma} for  joins between linear subspaces 
spanned by the points and the {\it secant varieties} of the unique \emph{rational normal curve} of degree $n$ determined by the collection of 
$n+3$ points (Lemma \ref{base locus lemma II}). 
For $n$ even, $X_{n+3,(0)}$ is isomorphic in codimension one to the 
$n$-dimensional Fano variety of $(\frac{n}{2}-1)$-planes 
in a smooth complete intersection of two quadrics in $\PP^{n+2}$,
studied recently in \cite{AC}. 
Araujo and Massarenti \cite{am} recently gave an explicit log Fano structure to the blow-up of $\PP^n$ 
in up to $n+3$ points in general position;  they do so by studying the blow-up $X^\sigma_{n+3,(n-2)}$
of $\PP^n$ along the joins mentioned above.

In this article, we prove a number of results in birational geometry.

The \emph{log abundance} conjecture is one of the main open questions in the birational classification of higher dimensional algebraic varieties. 
It predicts that for every pair given by a projective variety $X$ and an effective divisor $\Delta$ on $X$ with log canonical singularities, if the adjoint bundle $K_X+\Delta$  is nef, then it is semi-ample (cf. Conjecture \ref{abundance}). It is known to hold for surfaces and threefolds (see \cite{KMM}) and for fourfolds with
positive Kodaira dimension (see \cite{Fk}). Moreover, it was recently proved to hold for rationally connected varieties (see 
\cite{gon} and \cite{HMX}), otherwise very little is known in higher dimension.
A precise classification of divisors   on the blow-up of $\PP^n$ in collections of points in general position, for which abundance holds, is not easily deduced from the work contained in these references.
In this paper we construct explicitly infinite families of log pairs $(X_{s,(0)},D)$, with $D$ effective, for which  abundance holds, see Theorems \ref{abundance Q} and \ref{abundance Q n+3}.

As explained in \cite[Subsection 6.3]{bdp1}, the {\it F-conjecture} predicting the nef 
cone of $\overline{\mathcal{M}}_{0,n}$ was the original motivation for the study  of the vanishing theorems. 
 A divisor on $\overline{\mathcal{M}}_{0,n}$ intersecting non-negatively the one
dimensional strata of $\overline{\mathcal{M}}_{0,n}$ is called $F$-nef. The F-conjecture proposed by Fulton, states that a 
divisor on $\overline{\mathcal{M}}_{0,n}$ is nef if and only if is F-nef.

 As noted above, the iterated blow-up $X_{n+2,(n-2)}$ of the projective space is identified  with the moduli space $\overline{\mathcal{M}}_{0,n+3}$.
Using vanishing theorems we prove that the conjecture holds for strict transforms of
linear systems in $\PP^n$ interpolating multiple points in general position, under the blow-up
 of their linear base locus, Theorem \ref{fulton nef}.

This paper is organized as follows.

In part \ref{part 1} we study properties of divisors on $X_{s,(0)}$, that have only linear base locus, and of their strict transforms in further blow-ups $X_{s,(n-2)}$. In Part \ref{part 2} we study all effective divisors on $X_{n+3,(0)}$, that contain linear and non-liner obstructions, and we study properties of their strict transform in $X^\sigma_{n+3,(n-2)}$, the iterated blow-up of all cycles of the base locus. 

In Section \ref{notations} we introduce the
 general construction, notation and some preliminary facts. 

Section \ref{gg section} studies semi-ampleness of divisors on 
$X_{s,(n-2)}$, the main result is Theorem \ref{gg theorem Xr}. 
Subsection \ref{vanishing question} contains a complete description of the base locus of only linearly obstructed linear systems in $\PP^n$ interpolating multiple points, Theorem \ref{base locus theorem}.

In Sections \ref{abundance section}  we prove that the (log) abundance conjecture
 holds for log pairs on $X_{s,(0)}$ that have only linear base locus, by taking  log resolutions  $X_{s,(n-2)}$, Theorem \ref{abundance Q}
 
In Section \ref{fulton}, as an application of the results contained in Section \ref{gg section}, 
we establish the F-conjecture  for a particular class of divisors, Theorem \ref{fulton nef}.

In Section \ref{notation n+3} we construct log resolutions $X^\sigma_{n+3,(n-2)}$ of effective divisors on $X_{n+3,(0)}$. 
In Subsection \ref{section vanishing cohomology n+3} we pose a number of questions 
about the cohomology of the strict transform of these divisors.

In  Section \ref{section abundance n+3} we prove that log abundance conjecture holds for pairs on $X_{n+3,(0)}$ that are not only linearly obstructed, by taking log resolutions  $X^\sigma_{n+3,(n-2)}$.


 \part[I]{Linearly obstructed divisors}\label{part 1}

\section{Notation, known results and conjectures}\label{notations}

Let $K$ be an algebraically closed field of characteristic zero.
 Let $\mathcal{S}=\{p_1,\dots,p_s\}$ be a collection of $s$ distinct points in $\PP^n_K$ and
let $S$ be the set of indices parametrizing $\mathcal{S}$, with $|S|=s$.
\begin{notation} (Interpolation problems)
Let \begin{equation}\label{linear system}
\ls:=\ls_{n,d}(m_1,\dots,m_s)
\end{equation}
denote the linear system of degree-$d$ hypersurfaces of $\PP^n$ with multiplicity at least $m_i$ at $p_i$, for $i=1,\dots,s$.
\end{notation}

\begin{notation}\label{blow-up in points}
 We denote by $X_{s,(0)}$ the blow-up of $\PP^n$ in the points $\mathcal{S}$ and by $E_i$ the exceptional divisor of $p_i$, for all $i$.
 The index $(0)$ indicates that the space $\PP^n$ is blown-up in $0$-dimensional cycles. The Picard group of $X_{s,(0)}$ is spanned by the class of a general hyperplane, $H$, and the classes of the
 exceptional divisors $E_{i}$, $i=1,\dots,s$. 
 \end{notation}

\begin{notation}\label{generality definition}
Fix positive integers $d, m_1, \ldots, m_s$ and define the following divisor on $X_{s,(0)}$:
  \begin{equation}\label{def D}
  dH-\sum_{i=1}^{s} m_i E_i\in\textrm{Pic}\left(X_{s,(0)}\right).
  \end{equation}
Let $D$ be a \textit{general divisor} in $|dH-\sum_{i=1}^{s} m_i E_i|$. 

For every $I\subset\mathcal{S}$, we set $|I|=:\rho+1$ and we define the integer
\begin{equation}\label{mult k} 
k_{I}=k_I(D):=\max \{0,m_{i_1}+\cdots+m_{i_{\rho+1}}-\rho d\}.
\end{equation}
that computes the multiplicity of containment of $L_I$  in $\Bs(|D|)$, see
{\cite[Proposition 4.2]{dp}}.
  \end{notation}

\begin{notation}\label{higher blow ups} 
For $-1\le r\leq \min\{s,n\}-1$, we denote by  $X_{s,(r)}$ the iterated
 blow-up of $\PP^n$ along the strict transform of the linear cycles $L_I$ 
of dimension less than or equal to $r$ spanned by sets of points $I\subset\mathcal{S}$, 
$|I|\le r+1$, with  $k_{I}>0$, ordered by increasing dimension. 
We further denote by $E_I$ the (strict transform of) the exceptional divisor in $X_{s,(r)}$ of the linear space $L_I$, for every such an $I$.

When $r=0$,  $X_{s,(0)}$ is the blow-up introduced in Notation \ref{blow-up in points}. If $r=-1$ then we adopt the convention
$X_{s,(-1)}:=\mathbb{P}^n$.
\end{notation}

We point out that if $r\ge 1$ the space $X_{s,(r)}$ depends on the divisor $D$, precisely on the integers 
$k_I$, but for the sake of simplicity we decided to omit this dependency from the notation.
If  $D$ is an effective divisor on $X_{s,(0)}$, the birational morphism $$X_{s,(r)}\rightarrow X_{s,(0)},$$ obtained as composition of blow-ups,
resolves the linear base locus of $D$
(see \cite[Section 4]{bdp1} and \cite[Section 2]{dp}).

\begin{remark}\label{linear cycles}
Abusing notation, we will denote by $L_I$ the linear subspace of $X_{s,(-1)}$ 
spanned by the set of points parametrized by $I$ as well as its strict transform in the
 blown-up spaces $X_{s,(r)}$, for every $r\ge0$. 
\begin{enumerate}
\item  For $|I|\geq r+2$, $L_I$ on $X_{s,(r)}$ represents a blown-up projective space of dimension $|I|-1$ along linear cycles of dimension at most $r$. 
\item If $|I|\leq r+1$, the strict transform of $L_I$ will be the exceptional divisor $E_I$ that is a product of blown-up projective spaces. The full description of this space and its intersection theory is explicitly given in \cite[Section 2]{dp}.
\end{enumerate} 
\end{remark}

\begin{notation}\label{D_r}
Denote by $D_{(r)}$ the strict transform of the divisor $D$ on $X_{s, (r)}$. One has
\begin{equation}\label{strict transform}
D_{(r)}:=dH - \sum_{\substack{I\subset\{1,\dots,s\}:\\0\le |I|\le r+1}} k_{I} E_{I},
\end{equation}
where the integers $k_I$ are defined in \eqref{mult k}.
\end{notation}

\begin{remark}\label{remark divisorial subtraction}
If $r=n-1$, then $D_{(n-1)}$ is the strict transform of  $D_{(n-2)}$ via $X_{s,(n-2)}\dashrightarrow X_{s,(n-2)}$, that is the divisor on $X_{s,(n-2)}$ obtained from $D_{(n-2)}$  by subtracting $k_{I(n-1)}$ times the  strict transform of the fixed 
hyperplane spanned by the points parametrized by $I(n-1)$, for any $I(n-1)\subset\{1,\dots,s\}$, namely those that are fixed components 
of $D_{(n-2)}$.
\end{remark}

\begin{notation}\label{bar r}
Set $\bar{r}$ to be the maximal dimension of the linear cycles of the
 base locus of $D$. We will also use $\tilde{D}:=D_{(\bar{r})}$ to denote the strict transform of $D$ under the blow-up of all its linear base locus (including subtraction of
hyperplanes for $\bar{r}=n-1$, as in Remark \ref{remark divisorial subtraction}). 
 \end{notation}

 \subsection{The cohomology of divisors on $X_{s,(r)}$}

In this section we discuss the cohomology spaces $H^i(X_{s,(r)},\mathcal{O}_{X_{s,(r)}}(D_{(r)}))$.
 To simplify notation we will abbreviate $\dim H^i(X_{s,(r)},\mathcal{O}_{X_{s,(r)}}(D_{(r)}))$ by $h^i(D_{(r)})$. 

 Let $s(d)$ denote the number of points with multiplicity $d$ of $D$, namely the number of $i$'s such that $m_i=d$ in \eqref{def D}.

\begin{theorem}[{\cite[Theorem 5.3]{bdp1}, \cite[Theorem 5.12]{dp}}]\label{extended vanishing} Let $\mathcal{S}$ be a collection of points of $\PP^n$ in  general position. Let $D_{(r)}$ be as in Notation \ref{D_r}. 
Assume that
\begin{equation}\label{cond vanishing}
\begin{split}
&0\leq m_i \ \forall i\in\{1,\dots,s\},\\
& m_i+m_j\leq d + 1, \ \forall i,j\in\{1,\dots,s\},\ i\ne j\ (\textrm{if } s\ge 2)\\ 
&\sum_{i=1}^{s} m_i\le nd+\left\{ \begin{array}{ll}
n & \textrm{ if } s\le n+1 \textrm{ and } d\ge2\\
1 & \textrm{ if } s\le n+1 \textrm{ and } d=1\\
1 & \textrm{ if } s= n+2\\
\min\{n-s(d),s-n-2\} & \textrm{ if }  s\ge n+3\end{array}\right..
\end{split}
\end{equation}
Then  $h^i(D_{(r)})=0$, for every $i\ne 0,r+1$.
Moreover, $h^i(\tilde{D})=0$ for every $i\geq 1$.
\end{theorem}

Theorem \ref{extended vanishing} states that if $D$ is a \emph{special} divisor on $X_{s,(0)}$, i.e. one for which the dimension of the 
first cohomology group does not vanish, then as long as it satisfies the  bounds on the coefficients
 \eqref{cond vanishing}, its  strict transform $\tilde{D}$ obtained after resolving the 
linear base locus is no more special, i.e.  it has vanishing higher cohomology groups.
In \cite{bdp1} the following question was posed, namely whether a similar statement   is true for all cycles $-$linear and not linear$-$ of $\Bs(|D|)$.

\begin{question}[{\cite[Question 1.1]{bdp1}}]\label{question}
Consider any effective divisor $D$ in the blown-up $\PP^{n}$ at general points. 
Let $\widetilde{X}$ be the smooth composition of blow-ups of $\PP^n$ along the 
(strict transforms of the) 
cycles of the base locus of $|D|$, ordered in increasing dimension. 
Denote by $\widetilde{{\mathcal{D}}}$ the strict transform of the general divisor of $\LL$  in $\widetilde{X}$.
  Does $h^i(\widetilde{X},\mathcal{O}_{\widetilde{X}}(\widetilde{\mathcal{D}}))$ vanish for all $i\ge 1$?
\end{question}

Notice that we adopt the notation  $\tilde{D}$ for the strict transform after 
the blow-up of the linear base locus and $\widetilde{\mathcal{D}}$
 for the strict transform after the blow-up of the whole base locus.  
An affirmative answer to Question \ref{question} would imply that 
$h^0(D)=\chi(\widetilde{X}, \mathcal{O}_{\widetilde{X}}(\widetilde{\mathcal{D}}))$, the Euler characteristic.
 This would give the dimension of the linear system $\ls$, answering the corresponding interpolation problem. 

We also pose the following slightly different question. 
\begin{question}\label{question2}
Let $X$ be  the blown-up $\PP^{n}$ at general points and let $D$ be an effective divisor on $X$. 
Let $(\widetilde{X},\widetilde{D})$ be obtained via a log resolution of the pair $(X,D)$.
   Does $h^i(\widetilde{X},\mathcal{O}_{\widetilde{X}}(\widetilde{\mathcal{D}}))$ vanish for all $i\ge 1$?
\end{question}

In Section \ref{gg section} for $s\leq n+2$ and $D$ effective, or for $s\geq n+3$ and $D$ effective having only linear base locus, Theorem \ref{question answer} establishes that $\widetilde{\mathcal{D}}$ and $\tilde{D}$ coincide, answering positively Question \ref{question} and Question \ref{question2}.


\section{Globally generated divisors on $X_{s,(r)}$}\label{gg section}

Recall that $D$ denotes a divisor on $X_{s,(0)}$ of the form \eqref{def D} and $D_{(r)}$, as in \eqref{strict transform}, is its strict transform in $X_{s,(r)}$. See Notations \ref{generality definition}, \ref{D_r} and \ref{bar r} .
If $\bar{r}$ is the dimension of the linear base locus of $D$, then we set $\tilde{D}=D_{(\bar{r})}$.

In this section we establish  the following result, that classifies globally generated divisors of the form $D_{(r)}$, for every $r\ge 0$.

For any $s\geq n+3$  we introduce the following integer (see also \cite[Theorem 2.2]{dp-pos1}):
\begin{equation}\label{b for l-va}
\begin{split}
& b_0:=\left\{ \begin{array}{ll}
\min\{n-1, s-n-2\}-1 & \textrm{ if } m_1=d-1 \textrm{ and }m_i= 1, i\ge 2 \\
 \min\{n,s-n-2\}-1 & \textrm{otherwise,}\\\end{array}\right.
\end{split}
\end{equation}
while for $s\leq n+2$, we define $b_0:=-1$

\begin{theorem}\label{gg theorem Xr}
Assume that $\mathcal{S}\subset\PP^n$ is a collection of points in  general position.
 Assume that 
 $s\le n+1$ or that $s\ge n+2$ and 
 $d$ is large enough, namely 
\begin{equation}\label{gg equations Xr large s}
\sum_{i=1}^s m_i -nd \le  b_0.
\end{equation}
Then for any $0\leq r\leq n-1$ the divisor $D_{(r)}$ on $X_{s,(r)}$ is globally generated  if and only if

\begin{equation}\label{gg equations Xr}
\begin{split}
&0\le m_i \le d, \ \forall i\in\{1,\dots,s\},\\  
&0\le (r+1)d-\sum_{i\in I} m_i,\ \forall I\subseteq\{1,\dots,s\},\ |I|=r+2.\\  
\end{split}
\end{equation}
\end{theorem}

\begin{remark}
Notice that if $s\le \left\lfloor\frac{r+2}{r+1}n\right\rfloor$, then
 the condition on the degree \eqref{gg equations Xr large s} is always satisfied. 
Hence in this range, Theorem \ref{gg theorem Xr} provides 
a complete classification of divisors $D$ on $X_{s,(0)}$ whose strict transform $D_{(r)}$  in $X_{s,(0)}$ is  globally generated.
\end{remark}

\begin{remark}\label{gg for Dn-1}
If $\bar{r}=n-1$, then $\tilde{D}$ is the strict transform of  $D_{(n-2)}$ 
obtained by subtraction of strict transforms of hyperplanes $L_{I(n)}\subset\PP^n$, 
$X_{s,(n-2)}\dashrightarrow X_{s,(n-2)}$, cfr. Remark \ref{remark divisorial subtraction}.
Therefore in the same assumptions of Theorem \ref{gg theorem Xr}, we have that $\tilde{D}$
is globally generated if and only if
\begin{equation}\label{gg equations Xn-1}
\begin{split}
&0\le m_i \le d, \ \forall i\in\{1,\dots,s\},\\  
&0\le nd-\sum_{i\in I} m_i,\ \forall I\subseteq\{1,\dots,s\},\ |I|=n+1.\\  
\end{split}
\end{equation}
\end{remark}

\vskip.3cm 

In order to prove the main result of this section, Theorem \ref{gg theorem Xr}, we need the following result. 

\begin{lemma}\label{vanishing prelim} 
In the above notation, assume 
\begin{equation}\label{gg equations Xr large}
\begin{split}
&\sum_{i=1}^{s} m_i\le nd+\left\{ \begin{array}{ll}
0 & \textrm{ if } s\le n+2,\\
\min\{n-s(d),s-n-2\} & \textrm{ if }  s\ge n+3.\end{array}\right. 
\end{split}
\end{equation}
Then for every exceptional divisor $E_I$ and any infinitely near point $q\in E_I$, we have
$$h^1(D_{(r)}\otimes\II_q)\leq h^1(D_{(r)}-E_I).$$
\end{lemma}

\begin{proof}
We prove the statement by induction on $n$. We refer to \cite[Theorem 2.11]{dp-pos1} for the proof of the statement for $n=2$.
 We will assume $n\ge 3$. Consider the short exact sequence
\begin{equation}\label{sequence}
0\stackrel{}{\rightarrow}D_{(r)}-E_I\stackrel{}{\rightarrow}D_{(r)}\otimes\II_q\stackrel{}{\rightarrow}D_{(r)}\otimes\II_q|_{E_I}\stackrel{}{\rightarrow}0.
\end{equation}
We claim that 
\begin{equation}\label{claim 0}
h^1(D_{(r)}\otimes\II_q|_{E_I})=0.
\end{equation} 
The statement will follow from this, by looking at the long exact sequence in cohomology associated with \eqref{sequence}.

We now prove the claim, namely that \eqref{claim 0} holds. In order to do this, we set $\rho+1:=|I|$ and
we introduce  the positive integer 
\begin{equation}\label{alpha}
\alpha(I):=|\{I(\rho+1): k_{I(\rho+1)}\geq 1,\  I(\rho+1)\supset I\}|,
\end{equation}
where $I(\rho+1)$ denotes an index set contained in $\{1,\dots, s\}$ of cardinality $\rho+2$ and the integer 
$k_{I(\rho+1)}$ is defined as in \eqref{expanded}.

We recall that $E_I$ is a product, 
whose  second factor is isomorphic to $X^{n-\rho-1}_{\alpha(I),(r-\rho-1)}$, the blown-up projective space of dimension $n-\rho-1$ along linear cycles up to dimension $r-\rho-1$ spanned by $\alpha(I)$ points. We refer to \cite[Lemma 2.5]{dp} for details.
 Moreover, let $F$ 
be a divisor on the blow-up of $\PP^{n-\rho-1}$ at $\alpha(I)$  points in linearly general position, $X^{n-\rho-1}_{ \alpha(I),(0)}$, of the following form:
\begin{equation}\label{definition F}
F:=k_{I}h -\sum_{\substack { I(\rho+1)\supset I:\\k_{I(\rho+1)}\geq 1}} k_{I(\rho+1)}\cdot e_{I(\rho+1)|{I}}.
\end{equation}
We claim that the following holds.
\begin{enumerate}
\item The restriction of $D_{(r)}$ to $E_I$ is
$D_{(r)}|_{E_I}=(0, F_{(r-\rho-1)})$, where $F_{(r-\rho-1)}$ denotes the strict transform of $F$ in $X^{n-\rho-1}_{\alpha(I),(r-\rho-1)}$.
\item If $D$ satisfies the bound \eqref{gg equations Xr large}, so does $F$.
\end{enumerate}
The proof of \eqref{claim 0} follows from these two claims. Indeed
(1) implies that
$$\mathcal{O}_{X_{(r)}}(D_{(r)})\otimes\II_q|_{E_I}\cong \mathcal{O}_{X^{n-\rho-1}_{\alpha(I),(r-\rho-1)}} (F_{(r-\rho-1)})\otimes\II_{q_b},$$
where $q=(q_b,q_f)\in E_I$. This has vanishing first cohomology group, by induction on $n$, since $F$ satisfies \eqref{gg equations Xr large}, by (2).

\vskip.3cm

We first prove (1). The fact that the first factor of the restriction $D_{(r)}|_{E_I}$ is zero follows from the computation of normal bundles of the $E_I$'s and from intersection theory on $X_{s,(r)}$, see \cite[Section 4]{bdp1}  or \cite[Section 2]{dp}.
We now compute the second factor.  We have
\begin{align*}
E_I|_{E_I}& =(*, h)\\
 E_I|_{E_{I(j)}}& =(*, e_{{J}|I}), \textrm{ for all }J\supset I,
\end{align*}
where $*$ denotes the appropriate divisor on the first factor; here
we are only interested in the second factor.
The classes of the divisors $h$ and $e_{J|I}$ generate the Picard group of $X^{n-\rho-1}_{\alpha(I),(r-\rho-1)}$.
We can compute
$$\ \ \ \ \ D_{(r)}|_{E_I}=\left(0, k_{I}h -\sum_{ J\supset I} k_{J}e_{J|{I}}\right).$$
To conclude, we need to prove that the second factor on the right hand side of the above expression equals the strict transform $F_{(r-i-1)}$.
Notice first of all that for any integer $\tau\ge 1$, one has $|I(\tau+\rho+1)\setminus I|=\tau+1$. Hence it is enough to prove that 
$$\sum_{\substack {I(\rho+1)\supset I:\\
I(\rho+1)\subset I(\tau+\rho+1)}} k_{I(\rho+1)}-\tau k_I=k_{I(\tau+\rho+1)}.$$
By definition, the left hand side equals
\begin{equation*}\label{}
\begin{split}
&\sum_{\substack {I\subset I(\rho+1)\subset I(\tau+\rho+1)}} k_I+\sum_{j\in I(\tau+\rho+1)\setminus I}m_j-(\tau+1)d -\tau k_I=\\
&=(1+\tau)k_I+\sum_{j\in I(\tau+\rho+1)\setminus I}m_j-(\tau+1)d-\tau k_I\\
&=k_I+\sum_{j\in I(\tau+\rho+1)\setminus I}m_j-(\tau+1)d\\
&=k_{I(\tau+\rho+1)}.\\
\end{split}
\end{equation*}

\vskip.3cm 

We now  prove (2). We shall assume that $m_i\ge1$ for all $i\in \mathcal{S}$.
Notice  that for any index set $I(\rho+1)=I\cup\{j\}$, with $k_{I\cup \{j\}}\ge1$, we can write 
\begin{equation}\label{eq kI}k_{I(\rho+1)}-k_{I}=m_j-d.\end{equation}
Set  $A(I)=\{j\in S\setminus I: k_{I\cup \{j\}}\ge 1\}$. We have 
\begin{equation}\label{eq alpha}
\alpha:=|A(I)|=\alpha(I)\le s-\rho-1.
\end{equation}
We denote by $s_F(k_I)$ the number of sets $I(\rho+1)=I\cup\{j\}$ such that $k_{I(\rho+1)}=k_{I}$.
The divisor $F$ satisfies \eqref{gg equations Xr large} if and only if 
\begin{equation}\label{ineq for 2}
\begin{split}
\sum_{ j\in A(I)}  k_{I\cup\{j\}}\le &(n-\rho-1)k_I \\ 
  &+\left\{ \begin{array}{ll}
0 & \textrm{if } \alpha\le n-\rho+1,\\
\min\{n-\rho-1-s_F(k_I),\alpha-n+\rho-1\} & \textrm{if }  \alpha\ge n-\rho+2.\end{array} \right. \\
\end{split}
\end{equation}
Using \eqref{eq kI}, we can see that \eqref{ineq for 2} is equivalent to 
\begin{equation}\label{ineq for 2, 2}
\begin{split}
\sum_{ j\in A(I)}  m_j-\alpha d\le &(n-\rho-1-\alpha)k_I \\ 
  &+\left\{ \begin{array}{ll}
0 & \textrm{if } \alpha\le n-\rho+1,\\
\min\{n-\rho-1-s_F(k_I),\alpha-n+\rho-1\} & \textrm{if }  \alpha\ge n-\rho+2.\end{array} \right. \\
\end{split}
\end{equation}

For $\alpha\le n-\rho-1$, this  holds since $\alpha=|A(I)|$ and $m_j\le d$ so that the left hand side of \eqref{ineq for 2, 2} is a non-positive integer, while the right hand side is  non-negative.

For $\alpha\ge n-\rho$, the left hand side of \eqref{ineq for 2, 2} equals
$$
\sum_{j\in S} m_j - nd  - k_I-\sum_{j\in S\setminus (I\cup A(I))}m_j - (\alpha-n+\rho)d,
$$
therefore \eqref{ineq for 2, 2} is equivalent to 
\begin{equation}\label{ineq for 2, 3}
\begin{split}
&(\alpha-n+\rho)(k_I-d)-\sum_{j\in S\setminus (I\cup A(I))}m_j+\left(\sum_{j\in S} m_j-nd\right)
 \\
&\le 
\left\{ \begin{array}{ll}
0 & \textrm{if } \alpha\le n-\rho+1,\\
\min\{n-\rho-1-s_F(k_I),\alpha-n+\rho-1\} & \textrm{if }  \alpha\ge n-\rho+2.\end{array} \right.
\end{split}
\end{equation}
Notice that $| S\setminus (I\cup A(I))|=s-\rho-1-\alpha$. Using this, $k_I-d\le 0$ and $m_i\ge1$,
we obtain that the left hand side of \eqref{ineq for 2, 3} is bounded above by
\begin{equation}\label{ineq for 2, 4}
-(s-\rho-1-\alpha)+
\left\{\begin{array}{ll}
0 & \textrm{if } s\le n+2,\\
\min\{n-s(d),s-n-2\} & \textrm{if }  s\ge n+3.\end{array}
 \right.
\end{equation}
To conclude, it is enough to show that \eqref{ineq for 2, 4} is  bounded above by the expression in the right hand side of 
\eqref{ineq for 2, 3}. We do so in the following cases. 

Case (a) $s\le n+2$. The inequality \eqref{eq alpha} implies that $\alpha\le n-\rho+1$. 
The right  hand side of \eqref{ineq for 2, 3} is zero.  Moreover \eqref{ineq for 2, 4} equals
$-(s-\rho-1-\alpha)$ which is non-positive, therefore we conclude in this case.

Case (b) $s\ge n+3$ and $\alpha\le n-\rho+1$. As above the right  hand side of \eqref{ineq for 2, 3} is zero. 
Moreover \eqref{ineq for 2, 4} is bounded above by 
$-(s-\rho-1-\alpha)+s-n-2$ which is non-positive, so we conclude.

Case (c) $s\ge n+3$ and $\alpha\ge n-\rho+2$. We need to prove that 
\begin{equation}\label{ineq for 2, 5}
-(s-\rho-1-\alpha)+\min\{n-s(d),s-n-2\} \le
\min\{n-\rho-1-s_F(k_I),\alpha-n+\rho-1\}.
\end{equation}

Case (c.1) Assume that $k_I=d$. Notice that $m_i=d$ for all $i\in I$ and $s_F(k_I)$ corresponds
to the cardinality of the set $\{i\in \mathcal{S}\setminus I: m_i=d\}$. Therefore
$s(d)=s_F(k_I)+\rho+1$. Moreover, using $m_i\ge 1$ we also obtain $\alpha=s-\rho-1$.
Therefore $\min\{n-s(d), s-n-2\}= \min\{(n-\rho-1)-s_F(k_I),\alpha-n+\rho-1\}$ and this concludes the proof.

Case (c.2) Assume that $k_I-d\le -1$. Since $\alpha-n+\rho\geq 2$ then the left hand side of 
 \eqref{ineq for 2, 3} is bounded above by
\begin{align*}
 -(\alpha-n+\rho)-(s-\rho-1-\alpha)+\min\{n-s(d),s-n-2\}\le & \\
 -(\alpha-n+\rho)-(s-\rho-1-\alpha)+(s-n-2)=&-1.
\end{align*}
Therefore we conclude because the right hand side of \eqref{ineq for 2, 3} is positive by definition.

\end{proof}

\begin{proof}[Proof of Theorem \ref{gg theorem Xr}]
The case $r=0$ is proved in \cite[Corollary 2.4]{dp-pos1}. We will assume $r\ge 1$.
In order to prove that $D_{(r)}$ is globally generated, we will prove that
 $h^1(D_{(r)}\otimes\II_q)=0$, for all points $q\in X_{s,(r)}$.  
We distinguish the following two cases:
\begin{enumerate}
\item $q$ lies on some exceptional divisor $E_I$,
\item $q$ is the pull-back of a point outside of the union $\sum_{I\subseteq\{1,\dots,s\}, |I|\le r+1} L_I\subset\PP^n$.
\end{enumerate}

Case (1). Assume $q\in E_I$ and  write $\rho+1:=|I|$, $\rho\le r$. 
We want to prove that $h^1(D_{(r)}-E_I)=0$.  Then the conclusion will follow from  Theorem \ref{vanishing prelim}.

Reordering the points if necessary, we may assume that $1\in I$. 
Define the divisor $D'=D-E_1$ in $X_{s,(0)}$. 
 One can easily check that $D'$ satisfies the hypotheses of  Theorem
 \ref{extended vanishing}, therefore $h^1(D'_{(r)})=0$. 

Consider the set $\mathcal{J}$ of all indices $J$ of cardinality $1\le |J|\le r+1$ such that $1\in J$ and such that $\sum_{i\in J}m_i-|J|d\ge0$. 
We consider on $\mathcal{J}$ the \emph{graded lexicographical order}, namely   if $|J_1|<|J_2|$  then $J_1\prec J_2$, while if  $|J_1|=|J_2|$ we use the lexicographical order.

Notice that $I\in \mathcal{J}$ and that $D'_{(r)}=D_{(r)}-\sum_{J\in\mathcal{J}}E_J$.
In other words, $$D_{(r)}-E_I=D'_{(r)}+\sum_{\substack{J\in\mathcal{J}:\\ J\ne I}}E_J.$$
We can obtain $D'_{(r)}$ as the residual of iterative applications of short exact sequences starting from $D'_{(r)}+\sum_{J\in\mathcal{J}, J\ne I}E_J$ by restrictions to the exceptional divisors $E_J$, with $J\in\mathcal{J}$, $J\ne I$, following the order on $\mathcal{J}$. 
 More precisely, we start from $D'_{(r)}+\sum_{J\in\mathcal{J}, J\ne I}E_J$ and we restrict it to $E_1$,  
then we take the kernel of the so obtained exact sequence and we restrict, iteratively, to all $E_J$'s, with $J\in \mathcal{J}\setminus\{I\}$, $|J|=2$, then to all $E_J$'s, with $J\in \mathcal{J}\setminus\{I\}$, $|J|=3$, etc.
We claim that each restricted divisor has vanishing first cohomology group. This gives a proof of the statement, since the last kernel, that is $D'_{(r)}$, has vanishing first cohomology too, by Theorem \ref{extended vanishing}.

To prove the claim, recall that each exceptional divisor $E_J$ is the product of two blown-up projective spaces of dimension $|J|-1$ and $n-|J|$ respectively, see \cite[Lemma 2.6]{dp}. In particular the first component is isomorphic to $\PP^{|J|-1}$ blown-up along linear cycles (see Remark \ref{linear cycles}). Let us denote by $h$ the class of a general hyperplane on the first component and by $e_{J'}$ the class of the restriction of the exceptional divisors $E_{J'}$, namely $E_{J'}|_{E_J}=(e_{J'},0)$, for all $J'\subset J$. Let
$\Cr(h)$ be the proper transform of the \emph{standard Cremona transformation}
 of the hyperplane class $h$ on $\PP^{|J|-1}$, i.e. 
\begin{equation}\label{cremona}
\Cr(h)=(|J|-1)h-\sum_{\substack{J'\subset J:\\ |J'|<|J|-1}}(|J|-|J'|-1)e_{J'}.
\end{equation}
We have
\begin{equation*}
\begin{split}
D'_{(r)}|_{E_{J}}&=(0,*)  \textrm{ for every } E_{J},\\
{E_{J}}|_{E_{J}}&=(- \Cr(h),*), \\ 
{E_{J'}}|_{E_{J}}&=(0,*)  \textrm{ for all } J'\supset J,
\end{split}
\end{equation*}
where we use $*$ to denote the appropriate
divisor on the second factor. See \cite[Sect. 2,3]{dp} for details.

Therefore each of the above restrictions is 
$$\left(D'_{(r)}+\sum_{\substack{J'\in\mathcal{J}:\\ J'\ne I, J\prec J'}}E_{J'}\right)|_{E_{J}}=(-\Cr(h), *).$$ It  has vanishing first cohomology group by \cite[Theorem 3.1]{dp}. 
This concludes the proof of Case (1).

\vskip.3cm 

Case (2). In this case $q$ is the pull-back of a point $q'\in\PP^n\setminus\bigcup_{I\subset\{1,\dots,s\}, |I|\le r}L_I$. 
We prove the claim by induction on $n$. The case $n=1$ is obvious. Assume $n\ge2$.
We distinguish two subcases.

\vskip.3cm 

Case (2.a). 
Let us assume first that the points $p_1, \ldots, p_s, q'$ are not in linearly general position. If $s\ge n$, $q'$ lies on a hyperplane $H$ of $\PP^n$ spanned by $n$ points of 
$\mathcal{S}$. Reordering the points if necessary, assume that $q'\in H:=\langle p_1,\dots,p_n\rangle$.
If $s< n$, let $H$ be any hyperplane containing $\mathcal{S}\cup\{q'\}$.
Let $\bar{H}$ denote the pull-back of $H$ on $X_{s,(r)}$. It is isomorphic to the space $\PP^{n-1}$ blow-up along linear cycles of dimension up to $\min\{r,n-2\}$, 
spanned by $\bar{s}:=\min\{s,n\}$ points in general position, that we may denote by $\bar{H}\cong X^{n-1}_{\bar{s},(r)}$ as in Remark \ref{linear cycles}, part (1). As a divisor, we have 
\begin{equation}\label{bar H r}
\bar{H}=H-\sum_{i=1}^{\bar{s}-1} E_i-\sum_{\substack {I\subset \{1,\ldots, \bar{s}\}:, \\ 1\leq |I|\leq \min\{r,n-2\}}}E_{I}.
\end{equation}

Consider the restriction exact sequence of line bundles
\begin{equation}\label{restriction hyperplane r}
0\to D_{(r)}-\bar{H}\to D_{(r)}\otimes \II_{q} \to (D_{(r)}\otimes \II_{q})|_{\bar{H}}\to0.
\end{equation}
The restriction, $D_{(r)}|_{\bar{H}}\otimes \II_{q}$, is a \textit{toric divisor} on the blown-up space $\bar{H}\cong X^{n-1}_{\bar{s},(r)}$, with a point, $q$, in possible special linear configuration with the other points (see \cite[Section 5.3]{dp} for a short introduction on toric divisors on blown-up projective spaces). 
We conclude that it has vanishing first cohomology group by induction on $n$.
The kernel also has vanishing first cohomology because its base locus possibly contains linear cycles with multiplicity one, see \cite[Theorem 1.5]{dp}.
We  conclude that $h^1(D_{(r)}\otimes \II_{q})=0.$

\vskip.3cm 

Case (2.b). 
Let us assume now that the points $p_1, \ldots, p_s, q'$ are in linearly general position.
If $s\ge n-1$, let $H$ denote the hyperplane $\langle p_1,\dots,p_{n-1},q'\rangle$.  If $s< n-1$, let $H$ be any hyperplane containing
 $\mathcal{S}\cup\{q'\}$. In both cases such an $H$ exists by the assumption that
 $\mathcal{S}$ is a set of points in  general position. Let $\bar{H}$ denote the pull-back of $H$ on $X_{s,(r)}$, as in \eqref{bar H r} and consider the corresponding restriction sequence as in \eqref{restriction hyperplane r}.

As in Case (2.a), we conclude by induction on $n$ and by noticing that the kernel has only possibly simple linear obstructions.

\vskip.3cm 

Assume now that one of the inequalities in \eqref{gg equations Xr} does not hold.  We claim that $D_{(r)}$ is not globally generated. Indeed, if $m_i\geq d+1$ then the divisor $D$ is not effective therefore $D_{(r)}$ is not globally generated.
If $m_i\leq -1$ then the divisor $E_i$ is in the base locus of $D_{(r)}$. If $k_I\geq1$ 
for some $I$ such that $|I|=r+2$ and $r\leq n-1$, then the divisor $D_{(r)}$ contains in its base locus the 
strict transform of the linear cycle $L_{I}$ by Lemma \ref{base locus lemma II}, therefore is not globally generated.

\end{proof}

\begin{remark}
The strict transform $\tilde{D}=D_{(\bar{r})}$ of $\ls$ is base point free if \eqref{gg equations Xr} is
 satisfied with $r=\bar{r}$. 
\end{remark}

\subsection{Vanishing cohomology of strict transforms}\label{vanishing question}
In this section we will determine the base locus and their intersection multiplicity for divisors satisfying condition \eqref{gg equations Xr large}. Furthermore, Theorem \ref{question answer} answers Question 1.1 posed in \cite{bdp1} in this range.

Recall that a line bundle is globally generated if and only if the associated linear system is base point free. 
We are now ready to prove that Theorem \ref{gg theorem Xr} implies a complete description of the 
base locus of all non-empty linear systems in 
$\PP^n$ of the form $\ls=\ls_{n,d}(m_1,\dots, m_s)$ \eqref{linear system} 
that are only linearly obstructed.

\begin{theorem}[Base locus of linear systems]\label{base locus theorem}
The base locus of the divisor $D_{(r)}$ on $X_{s,(r)}$, strict transform of 
$D=dH-\sum_{i=1}^{s}m_iE_i$ satisfying \eqref{gg equations Xr large s}, namely
$$
\sum_{i=1}^s m_i\le nd+b_0,
$$
 is the formal sum \begin{equation}\label{formal sum linear base locus}
\sum_{\substack{I\subseteq\{1,\dots,s\}:\\ |I|\ge r+1}}k_I L_I\in A^\ast(X_{s,(r)}).
\end{equation}
In particular, if $s\le n+2$, the sum \eqref{formal sum linear base locus} with $r=-1$, describes the base locus of all non-empty linear systems $\ls_{n,d}(m_1,\dots,m_s)$, while if $s\geq n+3$ the sum \eqref{formal sum linear base locus} with $r=-1$, describes the base locus of non-empty linear systems satisfying the bound 
\eqref{gg equations Xr large s}.
\end{theorem}
\begin{proof}
By {\cite[Lemma 4.1]{bdp3}} (cf. also Lemma \ref{base locus lemma II} below),  each linear subspace $L_{I}$ with $k_{I}>0$ is a base locus cycle for $\ls$ and  $k_{I}$ is its exact multiplicity of containment. Therefore 
$$ \bigcup_{\substack{I\subseteq\{1,\dots,s\}: \\|I|\ge r+1}}k_I L_I\subset \Bs(|D_{(r)}|).$$
By Theorem \ref{gg theorem Xr}, the strict transform $D_{(r)}$ of an element of $\ls$ is base point free as soon as no higher dimensional cycle, i.e. no $(r+1)$-plane, is contained in the base locus, namely when $k_{I}\le0$ for all $I\subseteq\{1,\dots,s\}$ of cardinality $r+1$. In particular, if $\bar{r}$ is the dimension of the linear base locus of $D$, then  $\tilde{D}=D_{(\bar{r})}$ is base point free. 
Since the total transform of $D_{(r)}$ in $X_{s,(\bar{r})}$ equals 
$$\tilde{D}+ \sum_{\substack{I\subseteq\{1,\dots,s\}: \\|I|\ge r+1}} k_ I E_I,$$ 
one concludes that the base locus of $D$ is supported only along linear cycles
$$\Bs(|D_{(r)}|)\subset \bigcup_{\substack{I\subseteq\{1,\dots,s\}: \\|I|\ge r+1}}k_I L_I.$$
\end{proof}

\begin{remark}
The unique rational normal curve $C$ of degree $n$  through $n+3$ points of $\PP^n$ and its secant varieties 
$\sigma_t(C)$ were studied in \cite{bdp3} as
 cycles of the base locus of non-empty linear systems $\ls$. 
By {\cite[Lemma 4.1]{bdp3}}, condition \eqref{gg equations Xr large s} for $s=n+3$ says that neither $C$ nor $\sigma_t(C)$ are contained in the base locus of $\ls$ (cf. also Lemma \ref{base locus lemma II} below). Hence Theorem \ref{base locus theorem} states that  if $C$ is not in the base locus of $\ls$, then nothing else is, besides the linear cycles. 
\end{remark}

\begin{remark}\label{stable base locus}
Recall that for a line bundle $D_{(r)}$ as above, the \emph{stable base locus} is defined as $$\mathbb{B}(D_{(r)})=\cap_{m\in\mathbb{N}}\Bs(|mD_{(r)}|).$$ 
The obvious equality $k_I(mD)=mk_{I}(D)$, for any integer $m\ge1$, and Theorem \ref{base locus theorem} show that $\Bs(|mD_{(r)}|)=m\cdot\Bs(|D_{(r)}|)$. Therefore the base locus of $D_{(r)}$ is stable, namely $$\mathbb{B}(D_{(r)})=\Bs(|D_{(r)}|).$$
Because the stable base locus is invariant under taking multiples, we can extend this definition to the case of  $\mathbb{R}$-divisors, see \cite[Lemma 3.5.3]{BCHM}. In particular we can consider the stable base locus of $\epsilon D_{(r)}$, for any $\epsilon \in\mathbb{R},$ and obtain 
$$\mathbb{B}(\epsilon D_{(r)})= \mathbb{B}(  D_{(r)})= \Bs(|D_{(r)}|).$$

\end{remark}

We now can prove that the strict transform of $D$ in the iterated blow-up along 
its base locus has vanishing cohomology groups. This answers affirmatively Question \ref{question} .
\begin{theorem}\label{question answer}
Let $\mathcal{S}$ be a collection of points in  general position and $D$
 be any effective divisor on $X_{s, {(0)}}$. Assume the following holds:
\begin{equation}
\sum_{i=1}^{s}m_i-n d\leq
\left\{\begin{array}{ll}
0 & \textrm{if } s\le n+2,\\
\min\{n-s(d),s-n-2\} & \textrm{if }  s\ge n+3.\end{array}
 \right.
\end{equation}
Then Question \ref{question} has affirmative answer.
\end{theorem}

\begin{proof}
Since all cycles of the base locus of the divisors $D$ are linear by Theorem \ref{base locus theorem}, we conclude that $\widetilde{\mathcal{D}}$ equals $\tilde{D}$. The claims follow from Theorem \ref{extended vanishing}.
\end{proof}

\subsection{Log resolutions for divisors on blown-up projective spaces in points}

Let $D$ be any divisor on $X_{s, {(0)}}$ in the hypothesis of Theorem \ref{question answer} and let $\bar{r}$ be the maximum dimension of its linear base locus. 
 If the divisor $D$ has $\bar{r}<n-1$, Theorem \ref{question answer} implies that the map $X_{s, (\bar{r})}\stackrel{}{\rightarrow} X_{s, (0)}$ is a resolution of singularities of $D$.

Moreover, we recall that a \emph{log resolution} of the pair $(X_{s,(0)},D)$ is a birational morphism
$\pi:Y\to X_{s,(0)}$ such that the pair $(Y, \tilde{D})$ is \emph{log smooth}, where 
$\tilde{D}=\pi_\ast^{-1}D$ is the strict transform of $D$, namely such that 
$Y$ is smooth and the sum $\tilde{D}+\textrm{Exc}(\pi)$, where 
$\textrm{Exc}(\pi)$ is the sum of exceptional divisors of $\pi$, is simple normal crossing. 
For $D$ a general divisor of the linear system $|dH-\sum_{i=1}^sm_i E_i|$ as in
 Notation \ref{generality definition}, we obtain the following result.

\begin{corollary}\label{log resolutions}

If $D$ satisfies relation \eqref{gg equations Xr large s}, then
the pair $(X_{s,(\bar{r})},D_{(\bar{r})})$ is log smooth and the birational morphism 
$X_{s, (\bar{r})}\stackrel{}{\rightarrow} X_{s, (0)}$
is a \emph{log resolution} of the singularities of the pair $(X_{s,(0)},D)$.
\end{corollary}
\begin{proof}
The variety  $X_{s,(\bar{r})}$ is smooth.
Indeed at each step of the blow-up of $X$, 
$X_{s,(r)}\to X_{s,(r-1)}$, $r\le\bar{r}$, the center of the blow-up is a disjoint union of smooth 
subvarieties, namely the strict transforms of the liner cycles $L_{I(r)}$ with $k_{I(r)}>0$
 (see \cite[Subsection 4.2]{dp} for details). This, together with the fact that each exceptional divisor is smooth, also proves that the sum $\Exc(\pi)=\sum_{r=1}^{n-2}\sum_{I(r)}E_{I(r)}$ is simple normal crossing. 
 Moreover $D_{(\bar{r})}$ is base point free and its  support 
intersects transversally all exceptional divisors $E_{I(r)}$, by  Theorem \ref{base locus theorem}. Finally since $D$ is general, then $D_{(\bar{r})}$ is general too hence it is smooth by Bertini's theorem. 

\end{proof}

\section{Log abundance for linearly obstructed divisors on $X_{s,(0)}$}\label{abundance section}

In this section we  construct an infinite family of log canonical pairs  given 
by effective divisors on the blow-up of $\PP^n$ at $s$ points in general position, 
$X_{s,(0)}$, with arbitrary $s$. Moreover
we prove that the (log) abundance conjecture holds for these pairs, when $D$ is 
\emph{only linearly obstructed}, namely $D$ has only linear base locus.

In the case $s\le n+3$, the blown-up spaces $X_{s,(0)}$ are Mori dream spaces (see \cite{am,CT, Mukai} 
 for an explicit proof), 
therefore it is a well-known fact that all nef divisors are semi-ample. 
Moreover this statement holds for $s\le 2n$, see \cite[Theorem 3.2]{dp-pos1} and \cite[Proposition 1.4]{am}.
For $s\ge n+4$, the result is new (Section \ref{section abundance n+2}).

\begin{notation}

(\cite{KM})
Let $(X,\Delta)$ be a log pair, with $X$ a normal variety and 
$\Delta=\sum_j a_j\Delta_j$ a formal
$\mathbb{Q}$-linear combination of prime divisors. 
Let $\pi:Y\to X$ be a log resolution of $(X,\Delta)$, denote by
 $\tilde{\Delta}:=\pi_*^{-1}\Delta$ the strict transform of $\Delta$ and by 
$E_i$ the exceptional divisors.
Write 
$$K_Y+\tilde{\Delta}=\pi^*(K_X+\Delta)+\sum_i a(E_{i},X,\Delta) E_i,$$
 with $a(E_{i},X,\Delta)\in\mathbb{Q}.$
If $a_j\le 1$, then the \emph{discrepancy} of the pair
 $(X,\Delta)$ can be computed as
\begin{equation}\label{discrepancy}
\textrm{discrep}(X,\Delta)=\min_i\left\{a(E_{i},X,\Delta), \min_j\{1-a_j\},1\right\},
\end{equation}
see \cite[Corollary 2.32]{KM}.
The pair $(X,\Delta)$ is  said to be \emph{log canonical} (\emph{lc}) 
if  $\textrm{discrep}(X,\Delta)\ge-1$.

\end{notation}

\begin{conjecture}[{\cite[Conjecture 3.12]{KM}}]\label{abundance}
Let $(X,\Delta)$ be lc, $\Delta$ effective. Then $K_X+\Delta$ 
is nef if and only if it is semi-ample.
\end{conjecture}

As an application of the results of this paper, we prove that the log abundance conjecture 
holds for $(X_{s,(0)},\Delta)$, where $\Delta:=\epsilon D$ with
$0\le \epsilon\le1$ and $D\ge 0$ effective divisor on $X_{s,(0)}$,
by explicitly constructing  a log resolution of the pair.

We recall here that if $s\le n+1$ or $s\ge n+2$ and  \eqref{gg equations Xr large s} holds, the divisor $D$  is effective, by \cite[Theorem 5.3]{bdp1}. 
For a small number of points, $s\leq n+2$, or an arbitrary number
 of points with a bound on the coefficients of the divisors, we construct an infinite family of log canonical pairs  given by  effective divisors $D$
on $X_{s,(0)}$.

\begin{theorem}\label{abundance Q}
  Fix integers $n> 3$, arbitrary $s$. Let 
$D=dH-\sum_{i=1}^sm_iE_i$ be a general effective 
divisor on $X=X_{s,(0)}$ with $s\le n+2$  or with $s\ge n+3$  and  satisfying
\eqref{gg equations Xr large s}. 
For every $\epsilon\in\mathbb{Q}^{\ge0}$, 
such that \begin{equation}\label{abundance condition}
\begin{split}
\ & \epsilon m_i \ge n-1,\ \forall i\in\{1,\dots,s\}\\
\ & \epsilon (m_i+m_j-d)\le n-3, \ \forall i,j\in\{1,\dots,s\}, \ i\neq j,
\end{split}
\end{equation}
the  pair $(X,\Delta)$ is lc.
\end{theorem}

Recall the classes of the canonical divisors on $X$,
$$
K_X=-(n+1)H+(n-1)\sum E_i,
$$ and consider the $\mathbb{Q}$-divisor 
$$K_X+\Delta=(\epsilon d-n-1)H-\sum_i (\epsilon m_i-n+1) E_i.$$

\begin{corollary}\label{abundance A}
In the same hypotheses of Theorem \ref{abundance Q} then
\begin{enumerate}
\item Conjecture \ref{abundance} holds for the pair $(X,\Delta)$, namely if $K_X+\Delta$ is nef then it is semi-ample. 
\item The canonical ring 
$$\bigoplus_{l\ge0}H^0(X,\mathcal{O}_X(l K_X+\lfloor l\Delta \rfloor))$$ 
is finitely generated.
\end{enumerate}
\end{corollary}
\begin{proof}
By Theorem \ref{abundance Q}, the pair $(X,\Delta)$ is lc. To conclude the first part, 
it is easy to see that under the condition \eqref{gg equations Xr large s}, then the divisor $K_X+\Delta$ is nef (equiv. semi-ample) 
if and only if the
conditions \eqref{abundance condition} of Theorem \ref{abundance Q} are verified.
This follows from \cite[Theorem 3.1]{dp-pos1}. 
In fact, if $s\le 2n$, the nef and the semi-ample cone coincide,
and the thesis follows trivially.
 Otherwise, if $s\ge 2n+1$,
 \eqref{gg equations Xr large s} implies that the hypotheses
of \cite[Theorem 3.1]{dp-pos1} are satisfied.  The first statement follows.

For the second statement see \cite[Section 3.13]{KM}.
\end{proof}

\begin{remark}\label{abundance R}
Let $0\le \epsilon\ll1$, $\epsilon\in\mathbb{R}$. Let $\Delta=\epsilon D\in N^1(X_{s,(0)})_\mathbb{R}$  be an $\mathbb{R}$-divisor on $X_{s,(0)}$ satisfying the assumption of Theorem \ref{abundance Q}.
Then Conjecture \ref{abundance} holds for the pair $(X,\Delta)$. This follows from Theorem \ref{abundance Q} and Remark \ref{stable base locus}.
\end{remark}

\begin{remark}
In the notation of Theorem \ref{abundance Q}, if $s\le n+2$  or $s\ge n+3$ and  condition \eqref{gg equations Xr large s} is satisfied, then $D$ is effective and 
only linearly obstructed, see \cite[Theorem 5.3]{bdp1}. In this case a log resolution of the corresponding pair is given by the iterated blow-up of the linear base locus, see Corollary \ref{log resolutions}.

The case $s=n+3$ is the first case where non-linear obstructions appear for divisors violating \eqref{gg equations Xr large s}. The base locus of effective divisors $D$ was studied in \cite{bdp3} (cf. Lemma \ref{base locus lemma II} below). In this case a log resolution of any pair given satisfying \eqref{abundance condition}, 
 will be constructed in Section \ref{construction log resolution} by means of the iterated blow-up along the subvarieties, linear and non-linear, contained in the base locus.

\end{remark}

\subsection{Proof of Theorem \ref{abundance Q}, only linearly obstructed case}
\label{section abundance n+2}

Let $s$ be an arbitrary integer.
Set $X:=X_{s,(0)}$ and $Y:=X_{s,(n-2)}$. Let 
$$D=dH-\sum_{i=1}^{s}m_i E_i\ge0$$
be a divisor as in Notation \ref{generality definition}.
Assume moreover that $D$ is only linearly obstructed, 
namely that condition \eqref{gg equations Xr large s} is verified.

\begin{remark}\label{movability}
Notice that under the assumption \eqref{abundance condition} the divisor $D$ is irreducible. We recall that $D$ represents a general member of the linear system $|D|$ and the assumptions above force $D$ to have no divisorial components.
Indeed, observe that no (strict transform of) hyperplane spanned by $n$ points is contained in the base locus of $D$. In fact, we can see first of all that \eqref{abundance condition} 
implies that for every $i=1 \dots s$, 
\begin{equation}\label{extra nef condition}
\epsilon (m_i-d) \le -2.
\end{equation}
Now,  if $I=I(n-1)$ is an index set of cardinality $n$, we can compute that the multiplicity of containment of the corresponding hyperplane is zero as follows:
\begin{align*}
\sum_{i\in I}m_i-(n-1)d&=\left(\sum_{i\in I\setminus\{i_1,i_2\}}m_i-(n-2)d\right)
+\left(m_{i_1}+m_{i_2}-d\right)\\
& \le -2(n-2)+(n-3)\\ & \le0.
\end{align*}
\end{remark}

Using the notation introduced in Subsection \ref{higher blow ups}, 
let $\pi:Y\to X$ be the composition of blow-ups  of $X$ along lines, 
then planes etc., up to codimension-$2$ linear cycles $L_I$, $I\subset\{1,\dots,s\}$
 for which $k_I>0$, see \eqref{mult k} for the definition of $k_I$.
As in Notation \ref{D_r}, the strict transform of $D$, is given by
$$
\tilde{D}=dH-\sum_{i}m_i E_i-\sum_{r=1}^{n-2}\sum_{I(r)}k_{I(r)}E_{I(r)}.
$$

\begin{proposition}\label{log resolution n+2}
In the above notation, the map $\pi:Y\to X$ is a log resolution of $(X, \Delta)$, for every
$\epsilon\ge0$.
\end{proposition}
\begin{proof} It follows from Corollary \ref{log resolutions}.

\end{proof}

Recall the class of the canonical divisors on $Y$:
$$
K_Y=-(n+1)H+(n-1)\sum E_i+\sum_{r=1}^{n-2}(n-r-1)\sum_{I(r)}E_{I(r)}.
$$
Here we abuse notation by denoting by $H$ the hyperplane class in both $X$ and $Y$; 
similarly by abuse of notation we denote by $E_i$ an exceptional divisor in $X$
and its strict transform in $Y$.

\vskip.3cm 

For $0\le \epsilon<1$, 
consider the pairs $(X,\Delta)=(X_{s,(0)},\epsilon D)$ and  $(Y,\tilde{\Delta})=(X_{s,(n-2)},\epsilon\tilde{D})$ and write
$$
K_Y+\tilde{\Delta}=\pi^*(K_X+\Delta)+
\sum_{1\le r\le n-2}(n-r-1-\epsilon k_{I(r)})E_{I(r)}.
$$
 We have $$a_i(E_{I(r)},X,\Delta)=n-r-1-\epsilon k_{I(r)},$$
for any $I(r)$ such that $1\le r\le n-2$ (cf. \cite[Lemma 2.29]{KM}).
Therefore $\textrm{discrep}(X,\Delta)\ge-1$ if 
\begin{equation}\label{lc condition}
\epsilon k_{I(r)} \le n-r, \  \forall I(r), \ 1\le r\le n-2.
\end{equation}

We are now ready to prove the first main result of this section.

\begin{proof}[Proof of Theorem \ref{abundance Q}, only linearly obstructed case]

We are now ready to prove Theorem \ref{abundance Q} for divisors with only linear obstructions.
By Proposition \ref{log resolution n+2}, $(Y,\tilde{\Delta})$ 
is log smooth and $\pi:Y\to X$ is a log resolution of $(X,\Delta)$.
We are going to prove that the pair is $(X,\Delta)$ is lc.

We  prove that \eqref{lc condition} 
is implied by  \eqref{abundance condition}, second line, 
and by \eqref{extra nef condition}. 
Indeed, take for instance $r=2$. If $k_{I(2)}=0$ the statement 
is obvious. Assume that $k_{I(2)}>0$. 
Write $I(2)=\{i_1,i_2,i_3\}$.
We have $\epsilon k_{I(2)}=\epsilon((m_{i_1}-d)+(m_{i_2}+m_{i_3}-d))\le 
-2+(n-3)\le n-2$. The same holds for $r\ge 3$.
\end{proof}


\section{On the F-conjecture} \label{fulton}

In this section we discuss an application of our results.
Let $\overline{\mathcal{M}}_{0,n}$ be the \textit{moduli space of stable rational 
curves with $n$ marked points}. 
For $n= 5$, $\overline{\mathcal{M}}_{0,n}$ is a del Pezzo surface and it has the property of being a
Mori dream space. Hu and Keel in \cite{HuKeel} showed that $\overline{\mathcal{M}}_{0,6}$
 is a log Fano threefold, hence a Mori dream space;  
Castravet computed its Cox ring in \cite{CA}. 
For $n\geq 10$,  $\overline{\mathcal{M}}_{0,n}$ is known to not be a Mori dream space, see \cite{CT2,karu, LKH}.

We recall here the \emph{F-conjecture} on the nef cone of $\overline{\mathcal{M}}_{0,n}$ proposed by Fulton. 
The elements of the $1$-dimensional boundary strata on  $\overline{\mathcal{M}}_{0,n}$ are called \textit{F-curves}.
A divisor intersecting non-negatively all F-curves is said to be \textit{F-nef}. 
The F-Conjecture states that a divisor on $\overline{\mathcal{M}}_{0,n}$ is nef if and only if it is F-nef. This conjecture was proved for $n\leq 7$ in by Keel and McKernan \cite{KeMc}.

\subsection{Preliminaries and notation}
Let $\mathcal{I}$ be a subset of $\{1, \ldots, n+3\}$ with cardinality 
$2\leq |\mathcal{I}|\leq n+1$ and let $\Delta_{\mathcal{I}}$ denote a
 boundary divisor on $\overline{\mathcal{M}}_{0,n+3}$. Here, $\Delta_{\mathcal{I}}$
 is the divisor parametrizing curves with one component marked by the elements 
of $\mathcal{I}$ and the other component marked by elements of its complement, 
${\mathcal{I}}^c$, in $\{1, \ldots, n+3\}$. Obviously $\Delta_{\mathcal{I}}=\Delta_{{\mathcal{I}}^c}$.

We recall that, for any $1\leq i\leq n+3$,  the tautological class $\psi_i$ is defined as the first Chern class of
the cotangent bundle, $c_1(\mathbb{L}_i)$, where $\mathbb{L}_i$ is the line bundle on $\overline{\mathcal{M}}_{0,n+3}$
such that over a moduli point $(C, x_1, \ldots, x_{n+3})$ the fiber is the cotangent space to $C$ at $x_i$, $T_{x_i}^{*}C.$

 In the Kapranov's model given by $\psi_{n+3}$ , denote by $\mathcal{S}$ the collection of
 $n+2$ points in general position in $\PP^n$ 
 obtained by contraction of sections
 $\sigma_i$ of the forgetful morphism of the $n+3$ marked point. Denote by $S$ the 
set of indices parametrizing $\mathcal{S}$.

Further, denote by $\mathcal{X}_{n+2,(n-2)}$ the iterated
 blow-up of $\PP^n$ along the strict transforms of \textit{all} linear subspaces $L_I$ 
of dimension at most $n-2$ spanned by sets of points $I\subset\mathcal{S}$ with
$|I|\le n-1$, ordered by increasing dimension. Notice that, in the notation of Section \ref{higher blow ups}, for an effective divisor $D$, the iterated blown-up
space $X_{n+2, (n-2)}$ along the linear subspaces that are in the base locus of $|D|$ is a resolution of singularities of $D$ (see Corollary \ref{log resolutions}), so it depends on the divisor we start with. However, $\mathcal{X}_{n+2,(n-2)}$ depends only on the original set of $n+2$ points in $\PP^n$.

In \cite{Ka} Kapranov identifies the moduli space $\overline{\mathcal{M}}_{0,n+3}$ with the 
projective variety  $\mathcal{X}_{n+2, (n-2)}$, in the notation of Section
 \ref{notations}, by constructing birational maps from $\overline{\mathcal{M}}_{0,n+3}$ 
to $\PP^{n}$ induced by the divisors $\psi_i$, for any choice of $i$ with $1\leq i\leq n+3$.

 We recall that the Picard group of $\mathcal{X}_{n+2, (n-2)}$ is spanned by a general 
hyperplane class and exceptional divisors, $\Pic(\mathcal{X}_{n+2, (n-2)})= \langle  H, E_{J} \rangle$, 
where $J$ is any non-empty subset of $S$ with $1\leq |J|\leq n-1$. 

\begin{remark} In the Kapranov's model given by $\psi_{n+3}$, $\mathcal{X}_{n+2, (n-2)}$, there are $n+2$ Cremona transformations that are based on any subset of $n+1$ points of $S$, $S_i:=S \setminus \{p_i\}$. The $\psi_{i}$ classes with $i\neq n+3$ correspond to the image of the Cremona transformation of a general hyperplane class $H$, based on the set $S_i$, denoted by $\Cr_i(H)=\Cr(H)$ in \eqref{cremona}, while $\psi_{n+3}$ corresponds to $H$.
\end{remark}

Furthermore, we have the following identification:
\begin{equation}
 \label{boundary divisors}
\Delta_{\mathcal{I}}= \begin{cases}
 E_J, \qquad |\mathcal{I}|\leq n,\\
 H_J,  \qquad |\mathcal{I}|= n+1.
 \end{cases}
 \end{equation}
where $E_J$ is the strict transform of the exceptional divisor obtained by blowing-up
 the linear cycle spanned points of $J$, while $H_J$ is the strict transform of the 
hyperplane passing thought the points of $J$, namely
$$H_J:=H-\sum_{\substack {I\subset J:\\ 1\leq |I|\leq n-2}}E_I.$$

The \textit{$F$-curves on $\overline{\mathcal{M}}_{0,n+3}$} correspond to 
partitions of the index set  
$$\mathcal{I}_1\sqcup \mathcal{I}_2\sqcup \mathcal{I}_3\sqcup \mathcal{I}_4=
\{1,\ldots, n+3\}.$$ 
We remark that by definition, all subsets $\mathcal{I}_i$ are non-empty.
We denote by $F_{\mathcal{I}_1, \mathcal{I}_2, \mathcal{I}_3, \mathcal{I}_4}$ the class of the
corresponding F-curve. 
We have the following intersection table (see \cite{KeMc}). 
 \begin{equation}
 \label{intersection f-curves}
F_{\mathcal{I}_1, \mathcal{I}_2, \mathcal{I}_3, \mathcal{I}_4}\cdot \Delta_{\mathcal{I}}= \begin{cases}
 1 \qquad  \ \  \mathcal{I}= \mathcal{I}_i\sqcup \mathcal{I}_j, \textrm{ for some } i\neq j,\\
 -1  \qquad \mathcal{I}=\mathcal{I}_i, \textrm{ for some } i,\\
 0  \qquad \  \textrm{ otherwise}.
 \end{cases}
 \end{equation}

We first describe the F-conjecture in a Kapranov's model using the coordinates of the
 N\'eron-Severi  group $N^1(\mathcal{X}_{n+2, (n-2)})$.
Consider a general divisor on $\mathcal{X}_{n+2, (n-2)}$ of the form 
\begin{equation}\label{general div on M}
dH -\sum_{\substack {I\subset S:\\1\leq |I| \leq n-1}} m_{I}E_{I}.
\end{equation}
For a non-empty subset $I$ of the  points parametrized by  $S$ we define
 \begin{equation}
 \label{coefficient1}
a_{I}:= \begin{cases}
 0 \qquad |I|\geq n ,\\
 1  \qquad |I|\leq n-1.
 \end{cases}
 \end{equation}
For any partition of the set of $n+3$ points $G \sqcup J \sqcup L=S$, set
\begin{equation*}
\begin{split}
A_{G, J, L}:&=d-a_G\cdot m_G-a_J\cdot m_J - a_L\cdot m_L + a_{J \sqcup L}\cdot m_{J \sqcup L} \\
&+ a_{J \sqcup G}\cdot m_{J \sqcup G}+ a_{L \sqcup G}\cdot m_{L \sqcup G}.
\end{split}
\end{equation*}

\begin{example}\label{f curves1}
Consider $|G|=n$ then $J$ and $L$  consist of one element each, say $j$
 and respectively $l$. Then $A_{G,J,L}$ is independent of $G$ since
$A_{G,J,L}=d-m_j-m_l+m_{jl}$ for $n\geq 3$. 
 Whenever $|G|\leq n-1$ then $A_{G,J,L}$ depends on all three subsets.
\end{example}

Moreover, for any two non-empty subsets of $S$, $I$ and $J$, set
 \begin{equation}
 \label{coefficient2}
b_{I\sqcup J}:= \begin{cases}
 0 \qquad |I|+|J|\geq n,\\
 1  \qquad |I|+|J|\leq n-1.
 \end{cases}
 \end{equation}
For any partition $I \sqcup G \sqcup J \sqcup L=S$ set
$b_{I \sqcup J \sqcup L}:=b_{I \sqcup (J \sqcup L)}$, as defined in \eqref{coefficient2}, and
\begin{equation*}
\begin{split}
B_{I, G, J, L}:&=m_{I}- b_{ I \sqcup G}\cdot m_{I \sqcup G} - b_{I \sqcup J}\cdot m_{I \sqcup J} - b_{I \sqcup L}\cdot m_{I \sqcup L} + \\
&+ b_{I \sqcup J \sqcup L}\cdot m_{I \sqcup J \sqcup L} + b_{I \sqcup G \sqcup J}\cdot m_{I \sqcup G \sqcup J}+ b_{I \sqcup G \sqcup L}\cdot m_{I \sqcup G \sqcup L}.
\end{split}
\end{equation*}

\begin{example}\label{f curves2} If $|I|+|G|=n$ then $J$ and $L$ consist each of one element, call $j$ and $l$ respectively. In this case
\begin{itemize}
\item
If $|I|=n-1$ then the subsets $G$, $J$ and $L$ consist of one element each and $B_{I, G, J, L}=m_{I}.$
\item
If $|I|=n-2$ then $B_{I, G, J, L}=m_{I}-m_{I\sqcup \{j\}}-m_{I\sqcup \{l\}}.$
\item
If $|I|\leq n-3$ then $B_{I, G, J, L}=m_{I}-m_{I\sqcup \{j\}}-m_{I\sqcup \{l\}}+m_{I\sqcup \{j\}\sqcup \{l\}}.$
\end{itemize}
Whenever $|I|+|G| \leq n-1$ then $B_{I, G, J, L}$ depends on the four subsets of the partition.
\end{example}
\begin{remark}
The number $A_{G, J, L}$ represents the intersection product between the divisor $D$ and the corresponding F-curve contained in a hyperplane divisor and $B_{I, G, J, L}$ represents the intersection product between the divisor $D$ and the F-curve contained in some exceptional divisor $E_I$.
\end{remark}
Using the identification of boundary divisors \eqref{boundary divisors} and the intersection table \eqref{intersection f-curves},
it  is easy to see that the following remark holds.

\begin{remark}[The cone of F-nef divisors]\label{f-nef}
 A divisor on $\mathcal{X}_{n+2, (n-2)}$ of the form \eqref{general div on M} is \textit{F-nef} if the following sets of inequalities hold:
\begin{equation}\label{f-nef ineq}
\begin{split}
A_{G, J, L}\geq 0, & \quad \textrm{ for any partition }G \sqcup J \sqcup L=S,\\
B_{I, G, J, L} \geq 0, &\quad \textrm{ for any partition } I \sqcup G \sqcup J \sqcup L=S.
\end{split}
\end{equation}
 This cone is often referred to as the \emph{Faber cone} in the literature, see e.g. \cite{GKM}.
\end{remark}

\begin{conjecture}[F-conjecture]\label{f-conjecture} A divisor on $\mathcal{X}_{n+2, (n-2)}$ 
of the form \eqref{general div on M} 
 is nef if and only if \eqref{f-nef ineq} holds.
\end{conjecture}

Take a general divisor with degree and multiplicities labeled as in  
\eqref{general div on M}. We will now describe general properties of the $F$-nef divisors in a Kapranov's model that are useful in computations.

\begin{lemma}\label{degree of a f-nef divisor}
Any F-nef divisor satisfies $d\geq m_I\geq 0$, for every $I\subset S$, and $m_I\geq m_J$, 
for every $I,J\subset S$ with $I\subset J$.
\end{lemma}

\begin{proof}
We claim that these inequalities follow from \eqref{f-nef ineq}. For $n=2$ the claim is obvious, hence we assume $n\geq 3$.
In fact, the following inequalities hold:
\begin{enumerate}
\item $m_{I}\geq 0$, for every non-empty set $I$ with $|I|=n-1$,
\item $m_{I}\geq m_{J}$, for every non-empty sets $I,J$ with $I\subset J$.
\end{enumerate}
Claim (1) follows from Example \ref{f curves2} and  Remark \ref{f-nef}.
To prove claim (2) we apply induction on $|I|$. For $i\neq j$ and $i,j \notin I$
 we introduce the following notations: $I_i:=I \sqcup \{i\}$ and $I_{ij}:=I \sqcup \{i,j\}$.
For the first step of induction consider the sets $I$ and $G$ with $|I|=n-2$ and $|G|=2$. For any $i\neq j$  one has, by \eqref{f-nef ineq},
that $$m_I - m_{I_i} -m_{I_j} \geq 0.$$ 
Therefore claim (2) follows from claim (1) for any $I$ with $|I|=n-2$.
If $|I|\leq n-3$ the claim follows using backward induction on $|I|$. Indeed, by Example \ref{f curves2} we have
$$m_I- m_{I_i} - m_{I_j} + m_{I_{ij}}\geq 0,$$
therefore 
$$m_I- m_{I_i} \geq m_{I_j} - m_{I_{ij}}\geq 0.$$
Since $|I_i|=|I|+1$ and $I_i\subset I_{ij}$, the induction hypothesis holds for $I_i$, so the claim follows.

To see that $d\geq m_I$ we use Example \ref{f curves1} and claim (2) to obtain
$$d\geq m_i+(m_j-m_{ij})\geq m_I.$$
\end{proof}

\subsection{The F-conjecture holds for strict transforms on $\overline{\mathcal{M}}_{0,n}$}
The main result of this section is Theorem \ref{fulton nef}, stating that the F-conjecture holds 
for all divisors on $\mathcal{X}_{n+2, (n-2)}$ that are strict transforms of an effective divisor on $X_{n+2, (0)}$.

In the notation of Section \ref{notations}, $D$ be any effective divisor on $X_{n+2, (0)}$ and let $\tilde{D}$ denote its strict transform on $X_{n+2, (n-2)}$. 
A general divisor on $\mathcal{X}_{n+2, (n-2)}$, for arbitrary coefficients $d$ 
and $m_I$, is of the form \eqref{general div on M} while the divisors $\tilde{D}$ 
have arbitrary coefficients $d$ and $m_i$ while $m_I:=k_I$ defined in \eqref{mult k} 
for any index $I$ with $|I|\geq 2$. We can consider $\tilde{D}$ a divisor on 
$\mathcal{X}_{n+2, (n-2)}$.

\begin{theorem}\label{fulton nef} 
Assume $m_i\geq 0$, for all $i\in S$, then Conjecture \ref{f-conjecture} holds for $\tilde{D}$ on $\mathcal{X}_{n+2, (n-2)}$.
\end{theorem}

\begin{proof}
To prove the claim, notice first  that the effectivity of $D$ implies $\sum_{i\in S}m_i\le nd$ and
 $\sum_{i\in I}m_i\le nd$, for all $I\subset S$ such that $|I|=n+1$. Moreover, since $m_i\ge 0$, we 
conclude by Theorem \ref{gg theorem Xr} and Remark \ref{gg for Dn-1}.

We proved that $\tilde{D}$ is globally generated. Therefore
 $\tilde{D}$ is nef and in particular F-nef.
\end{proof}

\begin{remark} Let $f:\overline{\mathcal{M}}_{0,n}\stackrel{}{\rightarrow} X_{n+2, (0)}$ be the Kapranov blow-up.
 For divisors of type $\tilde{D}$, the inequalities of Remark \ref{gg for Dn-1} correspond exactly to the curves generating the Mori cones of $\overline{\mathcal{M}}_{0,n}$, namely the F-curves of $\overline{\mathcal{M}}_{0,n}$ not contracted by $f$.
\end{remark}

\begin{corollary}
For divisors of the form $\tilde{D}\ge 0$ on $\mathcal{X}_{n+2, (n-2)}$, the three properties of being
$F$-nef, nef and globally generated are equivalent. 
\end{corollary}

\begin{remark}
The divisors in \eqref{general div on M} for which $m_I<k_I$ are not nef, as they 
 intersect negatively the class of a general line on the exceptional divisor $E_I$, for $2\leq|I|\leq n-1.$ 

The divisors $\tilde{D}$ with $k_{I}\geq 1$, for some set $I$ with $|I|\geq 2$ are not globally generated since they contract the exceptional divisors $E_{I}$.

\end{remark}

\begin{remark}
Studying divisors interpolating higher dimensional linear cycles, $L_{I}$ for  $|I|\ge2$, 
$m_I>k_I$, is a possible approach to the F-conjecture. Indeed, 
once the vanishing theorems are established by techniques developed in \cite{bdp1} and \cite{dp} they could be used for describing globally generated divisors or
 ample and nef cones of $\overline{\mathcal{M}}_{0,n}$. The description of ample divisors on $\overline{\mathcal{M}}_{g,n}$ is an important question originally asked by Mumford and conjectured by Fulton for $g=0$. In particular, {\cite[Conjecture 0.2]{GKM}}
 holds for $g=0$ if and only if it holds for any $g$.
\end{remark}


\part{Divisors obstructed by rational curves and their secants}\label{part 2}

\section{Notation and preliminary results}\label{notation n+3}

We recall here notations and results introduced in \cite{bdp3}.

\begin{notation}\label{RNC and secants}
It is classically known that there exists a unique rational normal curve $C$ of degree $n$ passing through $n+3$ general points of $\PP^n$\cite{veronese}.
Let $\sigma_t=\sigma_t(C)$ denote the $t$-th secant variety of $C$, namely the Zariski closure of the union of $t$-secant $(t-1)$-planes. In this notation we have $\sigma_1=C$. 
For every $I\subset\{1,\dots,n+3\}$ with $|I|=r+1$, $-1\le r \le n$,
let $$\J(L_I,\sigma_t)$$ be  the \emph{join} of the linear cycle $L_I$ and $\sigma_t$; this is a cone with vertex $L_I$.
We use the conventions $|\emptyset|=0$ and $\sigma_0=\emptyset$.
\end{notation}

The dimension of the variety
$\J(L_I,\sigma_t)$ is 
\begin{equation}\label{dimension join}
r=r_{I,\sigma_t}:=\dim \J(L_I,\sigma_t)=|I|+2t-1.
\end{equation}

Notice if $t=0$, then the join $\J(L_I,\sigma_0)$ is the linear cycle spanned by the points 
parametrized by $I$, $L_I$, while if $I=\phi$, then the join $\J(L_I,\sigma_t)=\sigma_t$
 is the secant variety $\sigma_t$.

\subsection{Strong base locus lemma for divisors on $X_{s+3,(0)}$}
\begin{lemma}[{\cite[Lemma 4.1]{bdp3}}]\label{base locus lemma II}
For any effective divisor $D$  as in \eqref{def D} and for $r_{I,\sigma_t}\leq n-1$, 
 the multiplicity of containment in  $\Bs(|D|)$
of the strict transform in $X_{s,(0)}$ of the subvariety $\J(L_I,\sigma_t)$
is the integer
\begin{equation}\label{expanded}
\ \ \ k_{I,\sigma_t}=k_{I,\sigma_t}(D):=\max\left\{0, t\sum_{i=1}^{n+3} m_i+\sum_{i\in I}m_i-((n+1)t+|I|-1)d\right\}.
\end{equation}
\end{lemma}

In particular, the formula \eqref{expanded} for a linear cycle $L_I$ reads
$$k_I=\sum_{i \in I} m_i - (|I|-1)d,$$
while, for the secant variety  $\sigma_t$, it reads
$$k_{\sigma_t}=t\sum_{i=1}^{n+3} m_i -((n+1)t-1)d.$$

We will now prove a stronger result that will play a crucial role in Section 
\ref{abundance section}, particularly in Proposition \ref{linear series of strict transform even}.

\begin{proposition}\label{infinitesimal base locus} Let $D$ be an effective divisor on $X_{s,(0)}$.
 Let $p$ be a point on the variety $\J(L_I,\sigma_t)$ that does not lie on  any smaller join
$\J(L_{I'},\sigma_{t'})\subsetneq \J(L_{I},\sigma_{t})$, $I'\subset I$, $t'\le t$.
Then $p$ is contained in the base locus of the general member of the linear system $|D|$ with multiplicity precisely equal to $k_{I,\sigma_t}$.
\end{proposition}

\begin{proof} We first prove the claim for $t=0$ when the join $\J(L_I,\sigma_0)$ is the linear cycle $L_I$ 
that is in the base locus of the general member of $|D|$ 
with multiplicity of containment equal to $k_I=k_I(D)$, see \cite[Proposition 4.2]{dp}. 

Let us assume, by contradiction, that there is a point $p$ of $L_I$ with multiplicity of containment at least $k_I+1$. 
Let $I(r)$ be the largest index set such that $p\in L_{I(r)}$. 
Lemma 4.2 of \cite{CDDGP} implies $r\geq 1$. 
We introduce the following notation:
$$K_{I(r)}=K_{I(r)}(D):=\sum_{i\in I(r)} m_i - rd.$$ 

Consider first the case when $K_{I(r)}\geq 1$ and note that $k_{I(r)}=K_{I(r)}$. For $r\geq 2$, let $I(r-2)$ be a subset of $I(r)$, of cardinality equal to $r-1$, and take $H_{I(r-2)}$ a general hyperplane passing through $p$  and all points of $I(r-2)$.
For $r=1$,  $H_{I(-1)}$ denotes the hyperplane passing through the point $p$. For every $r\ge1$, consider  the divisor $D'$, defined as follows, and denote by $d'$ and $m_i'$ its corresponding degree and multiplicities:
$$D':=D+K_{I(r)} H_{I(r-2)}.$$
By assumption, the general member of $|D'|$ contains the point $p$ with multiplicity at least $1$. We can compute the multiplicity of containment of $L_{I(r)}$ in the base locus of $D'$:
\begin{equation}
\begin{split}
K_{I(r)}(D')&=\sum_{i\in I(r)} m'_i - rd'\\
&=\sum_{i\in I(r)} m_i - rd+(r-1)K_{I(r)}-rK_{I(r)}\\
&=K_{I(r)}-K_{I(r)}\\ &=0.
\end{split}
\end{equation}

Consider now the case  $K_{I(r)}<0$. Let $H_{I(r)}$ denote a hyperplane containing 
all points of $I(r)$. Define the divisor
$$D'':=D-K_{I(r)} H_{I(r)}.$$
A similar computation shows the following
\begin{equation}
\begin{split}
K_{I(r)}(D'')&=\sum_{i\in I(r)} m''_i - rd''\\
&=\sum_{i\in I(r)} m_i - rd-(r+1)K_{I(r)}+rK_{I(r)}\\
&=K_{I(r)}-K_{I(r)}\\&=0.
\end{split}
\end{equation}

In the above cases we reduced to the case when $D$ is a divisor with $k_{I(r)}=0$ 
whose general member has a base point, $p$. 
We  now prove by induction on $r$ that this leads to a contradiction.

We  discuss separately the case $r=1$ as the first induction step. The line $L_{I(1)}$ is not contained in the base locus of $D$ by Lemma \ref{base locus lemma II} for $t=0$ and $r=1$. However, the intersection multiplicity between the line $L_{I(1)}$ and $D$ is negative, a contradiction.

In general, we assume  that the statement holds for linear cycles of dimension $r-1$ 
and we prove that it holds for linear cycles of dimension $r$. 
The point $p$ can not lie
on any smaller linear cycle contained in some $L_{I(r-1)}\subset L_{I(r)}$, by the induction assumption. Therefore, $p$ is a point inside the \emph{interior} 
of the cycle $L_{I(r)}$, namely $p\in L_{I(r)}\setminus\bigcup_{I(r-1)\subset I(r)}L_{I(r-1)}$. 
 It is easy to see that whenever $k_{I(r)}\geq 0$, then $k_J\geq 0$ for any subset $J$. 
This implies that for any subset $I(r-1)\subset I(r)$ of cardinality $r$, the divisor $D$ 
contains the linear cycle $L_{I(r-1)}$ in its base locus with multiplicity $k_{I(r-1)}\geq 0$. 
We consider $l$ a general line in $L_{I(r)}$ passing trough $p$.
We observe that the multiplicity of intersection between $l$ and the divisor $D$ is at most 
\begin{equation*}
\begin{split}
d-\sum_{\substack{I(r-1)\subset I(r)}}k_{I(r-1)}-1&=
d-\sum_{\substack{I(r-1)\subset I(r)}}\left(\sum_{i\in I(r-1)}m_i - (r-1)d\right)-1\\
&=r\left(\sum_{i\in I(r)} m_i - rd\right) - 1\\
&= -1.
\end{split}
\end{equation*}
Since the family of lines passing through the point $p$ covers the linear
 cycle $L_{I(r)}$ one obtains that $L_{I(r)}$ is in the base locus of the divisor $D$ 
that is a contradiction with  Lemma \ref{base locus lemma II}, since $k_{I(r)}=0$ (for $t=0$).

If $t\geq 1$,  the proof follows by an argument similar to the one used in \cite[Proposition 4.2]{dp}.

\end{proof}

\vskip.3em

\subsection{The blown-up space $X^\sigma_{n+3,(n-2)}$}
\label{section vanishing cohomology n+3}

In this section we extend Question \ref{question} to the case of effective divisors on
the blow-up of $\PP^n$ at $n+3$ general points. Take 
\begin{equation}\label{divisor n+3 points}
D=dH-\sum_{i=1}^{n+3}m_iE_i\ge0,
\end{equation}
 a divisor on $X_{n+3,(0)}$, the blown-up $\PP^{n}$ at $n+3$ base points, with $d,m_i\ge0$.

We will obtain $X^\sigma_{n+3,(n-2)}$ from  $X_{n+3,(0)}$ by iterated blow-up along the subvarieties that are contained in the base locus of $D$, that are the strict
transforms of $\J(L_{I},\sigma_{t})$, see Lemma \ref{base locus lemma II}.
The pairwise intersections of such subvarieties  with some constraints on the index sets parametrizing the vertices, 
is computed in \cite[Proposition 5.6]{am}. For the sake of completeness we state below this result in our notation.

\begin{proposition}[{\cite[Proposition 5.6]{am}}]\label{intersection of joins}
Let $I_1,I_2\subset\{1,\dots,n+3\}$ be index sets such that 
$I_1\cap I_2=\emptyset$. Let $t_1,t_2\ge0$ 
 be integers such that 
\begin{equation}\label{conditions intersection of joins}
\begin{split}
&r_{I_1,\sigma_{t_1}}=r_{I_2,\sigma_{t_2}}, \\ 
&r_{I_i,\sigma_{t_i}}\le n-1, \ \forall i\in\{1,2\},\\   
& 2r_{I_i,\sigma_{t_i}}\le 2n-(|I_1|+|I_2|),\ \forall i\in\{1,2\}.
\end{split}
\end{equation}
Then $$\J(L_{I_1},\sigma_{t_1})\cap\J(L_{I_2},\sigma_{t_2})=\bigcup_{J}\J(L_J,\sigma_{t_J}),
$$
where the union is taken over all subsets $J\subseteq I_1\cup I_2$ satisfying
$$2|(I_i\cup J)\setminus(I_i\cap J)|=|I_1|+|I_2|,\ \forall i\in\{1,2\},$$
and for every such $J$, $t_J$ is the integer defined by the following equation
$$
2r_{J,\sigma_{t_J}}=2r_{I_i,\sigma_{t_i}}-(|I_1|+|I_2|).
$$
\end{proposition}

\begin{notation}\label{Ysigma}
For the sake of simplicity let us denote $Y^\sigma:=X^\sigma_{n+3,(n-2)}$ and $X:=X_{n+3,(0)}$.
For any effective divisor $D$ of the form \eqref{divisor n+3 points},
let $$\pi^\sigma:Y^\sigma\to X$$ be the iterated blow-up of $X:=X_{n+3,(0)}$ 
 along (the strict transforms of) all varieties $\J(L_I, {\sigma}_{t})$, 
$t\ge 0$, $|I|\ge0$, such that 
 $r_{I, {\sigma}_{t}}\le n-2$ and $k_{I,\sigma_t}>0$ in increasing dimension, composed with
a contraction of the strict transforms of the divisors $\J(L_I, {\sigma}_{t})$ 
with $r_{I, {\sigma}_{t}}= n-1$ and $k_{I, {\sigma}_{t}}>0$.
\end{notation}

The latter divisors  were described in \cite[Section 3.2]{bdp3}.
The space $X^\sigma_{n+3,(n-2)}$ is constructed by Araujo and Massarenti in their recent article \cite[Section 5]{am} in order to give explicit log
Fano structures on  $X_{n+3,(0)}$ (see \cite[Propositions 5.8, 5.11]{am}). 

\begin{notation}
We denote by $E_{I,\sigma_t}$ the exceptional divisors, for 
all $t\ge 0$, $|I|\ge0$, $r_{I, {\sigma}_{t}}\le n-1$.

It is immediate to see, using \cite[Lemma 4.1]{bdp3},
that the strict transform on $Y^\sigma$ of $D$ is given by 
\begin{equation}\label{definition D sigma}
\tilde{D}^{\sigma}:=dH-\sum_{i}m_i E_i-\sum_{r=1}^{n-1}\sum_{\substack{I,t:\\ r_{I,\sigma_t}=r}}k_{I,\sigma_t}E_{I,\sigma_t}.
\end{equation}
We stress the fact that the space $Y^\sigma$ depends on the divisor $D$.

\end{notation}

\subsection{Conjectures on vanishing cohomology for divisors on $X^\sigma_{n+3,(n-2)}$}
\label{conjectures} 
Let $D$ be an effective divisor on $\PP^n$ blown-up in $n+3$ general points.

\begin{question}\label{question sigma}
Consider the divisor $\tilde{D}_{\sigma}$ defined in \eqref{definition D sigma}
as the strict transform of $D$ on $X^\sigma_{n+3,(n-2)}$. 
Does $h^i(\tilde{D}^{\sigma})$ vanish for all $i\ge 1$?
\end{question}

\subsubsection{Related questions}

Another challenge would be  to compute the Euler characteristic of  $\tilde{D}^{\sigma}$. 
In what follows we would like to propose a candidate for such a number, namely 
the so called \emph{secant linear virtual dimension} for 
linear systems of hypersurfaces of $\PP^n$ interpolating $n+3$ general points with 
assigned multiplicity, or equivalently of linear systems $|D|$. This number  was introduced in \cite{bdp3}.

\begin{definition}[{\cite[Definition 6.1]{bdp3}}]\label{new definition rnc}
Let $D$ be a divisor on $X_{n+3,(0)}$ as in \eqref{divisor n+3 points}.
The {\em secant linear virtual dimension} of $|D|$ is the number
\begin{equation}\label{RNC expected dimension}
\sldim(D):=\sum_{I,\sigma_t}(-1)^{|I|}{{n+k_{I,\sigma_t}-r_{I,\sigma_t}-1}\choose n},
\end{equation}
where the sum ranges over all indexes $I\subset\{1,\dots,n+3\}$ and $t$ such that 
$0\le t \le l+\epsilon$, $n=2l+\epsilon$  and $0\le|I|\le n-2t$. The integers $k_{I,\sigma_t}$ 
and $r_{I,\sigma_t}$ are defined in \eqref{expanded} and in \eqref{dimension join} respectively.
\end{definition}

\begin{conjecture}\label{conjecture chi sigma}
The Euler characteristic of  the divisor $\tilde{D}^{\sigma}$, 
defined in \eqref{definition D sigma} as the strict transform of $D$ in $X^\sigma_{n+3,(n-2)}$, is
$$
\chi(\tilde{D}^{\sigma})=\sldim(D).
$$
\end{conjecture}

\vskip.3cm 

The above questions are related to the \emph{dimensionality problem} for linear systems of divisors of
the form \eqref{def D} 
and in particular to the  Fr\"oberg-Iarrobino   conjectures 
\cite{Froberg, Iarrobino}, which give a predicted value for the Hilbert series of
an  ideal generated by $s$ general powers of linear forms  in the polynomial ring
 with $n+1$ variables. We refer to \cite[Section 2.1]{bdp3} for a more detailed  account on this.
In \cite{bdp3} the following conjectural answer to this problem was given in terms of Definition \ref{new definition rnc}.

\begin{conjecture}[{\cite[Conjecture 6.4]{bdp3}}]\label{conjecture}
Set $X=X_{n+3,(0)}$ and let $D$ be as in \eqref{divisor n+3 points}. Then 
$$h^0(X,\mathcal{O}_X(D))=\max\{0,\sldim(D)\}.$$
\end{conjecture}

Previous work \cite{bdp1,dp} contains a proof that
Question \ref{question sigma} admits an affirmative answer, as well as proofs of Conjectures 
\ref{conjecture chi sigma} and \ref{conjecture},
 for divisors satisfying the bound \eqref{cond vanishing}, namely those that do not
contain positive multiples of the rational normal curve of degree $n$ in the base locus, nor joins $\J(I,\sigma_t)$. 
The approach adopted was based on
 the study of the normal bundles of the exceptional divisors of linear cycles, $L_I$, and 
vanishing cohomologies of strict transforms
on $X_{s,(n-2)}$, the blow-up along the linear cycles, see  Theorem 
\ref{extended vanishing}. 

We believe that a proof of the above conjectures for the general case for $s=n+3$
would rely on the study of the normal
bundles to the joins $\J(I,\sigma_t)$. We plan to develop
this approach in future work.

On a more general note,  we would like to point out that the construction of $X^{\sigma}_{s,(n-2)}$ and 
Question \ref{question sigma} and Conjectures
\ref{conjecture} and \ref{conjecture chi sigma} could be generalised to $\PP^n$ 
blown-up in arbitrary number of points in linearly general position.

\section{Explicit proof of log abundance for effective divisors on $X_{n+3,(0)}$}\label{section abundance n+3}

In this section we give an explicit constructive proof of the log abundance conjecture for pairs $(X_{n+3,(0)},D)$ where $D$ is any effective divisor. 
We recall that effective divisors on $X_{s,(0)}$,
 with $s\le n+3$ were classified in \cite{bdp3}.
The case where $D$ has only linear base locus is covered in Theorem \ref{abundance Q}. In this section we complete the picture by studying divisors $D$ that contains also the rational normal curve of degree $n$ through the $n+3$ points, as well as possibly joins between linear cycles and its secants (see Notation \ref{RNC and secants}).

\begin{theorem}\label{abundance Q n+3}
Let $n> 3$ be the dimension of the space and $\epsilon\in\mathbb{Q}$, 
with $0\le \epsilon\ll1$. Let 
$D=dH-\sum_{i=1}^sm_iE_i$ be a general effective 
divisor on $X=X_{s,(0)}$ and assume that \eqref{abundance condition} holds.
Then the  pair $(X,\Delta)$ is lc.
\end{theorem}

Set $X:=X_{n+3,(0)}$ and let $D$ on $X$ be a divisor of the form
$$D=dH-\sum_{i=1}^{n+3}m_iE_i\ge0.$$
 Write $\Delta=\epsilon D$, for $0\le\epsilon\ll1$ and assume it satisfies condition \eqref{abundance condition} of Theorem \ref{abundance Q}.

\begin{remark}\label{movable'}
Note that,  \eqref{abundance condition}  implies that $D$ is big (in fact it lies in the interior of the effective cone, see \cite[Theorem 5.1]{bdp3}).
In fact we can use \eqref{low mult of joins} that gives
$$
\epsilon k_{I(n-2t),\sigma_t}\le -n-1<0,
$$
for all $I(n-2t),\ t\ge 1$.
Moreover one can similarly check that 
$$\epsilon k_{I(n-2t-1),\sigma_t} \le0,$$ 
for all $I(n-2t-1),\ t\ge 1$, hence $D$ is movable, see \cite[Theorem 5.3]{bdp3}.  
In particular the general element of $|D|$ is irreducible.
\end{remark}

\begin{remark}
The argument developed in Section \ref{section abundance n+2} 
applies also to the case $s=n+3$ and $k_C:=k_C(D)=0$ (see \eqref{expanded} 
for the definition of this number), namely to divisors $D$ on $X_{n+3,(0)}$
 that have only linear base locus and for which \eqref{gg equations Xr large s} is satisfied. Indeed under this condition and \eqref{abundance condition}, $D$ is movable (see \cite[Theorem 5.3]{bdp3}) and its strict transform on $X_{n+3,(n-2)}$,  the iterated blow-up along the linear cycles, is globally generated by Theorem \ref{gg theorem Xr}.
Moreover  the lc condition on the pair $(X,\epsilon\Delta)$
 can be verified in the same manner as in the case with \eqref{gg equations Xr large s}, 
as in  Section
\ref{section abundance n+2}.
\end{remark}

From now on we will assume that $k_C(D)\ge1$.

\subsection{Constructing a log resolution of $(X,\epsilon D)$}\label{construction log resolution}

The following result computes the  pairwise intersections of the joins $\J(L_I, {\sigma}_{t})$.

\begin{proposition}\label{center of blow-up is smooth}
\label{snc} 
Let $\J(L_{I_1},\sigma_{t_1})$ and $\J(L_{I_2},\sigma_{t_2})$ be any join 
varieties of the same dimension $r_{I_1,\sigma_{t_1}}=r_{I_2,\sigma_{t_2}}$. If $2r_{I_1,\sigma_{t_1}}\le n-1$, 

 then
 $$\J(L_{I_1},\sigma_{t_1})\cap\J(L_{I_2},\sigma_{t_2})=\bigcup_{J}\J(L_J,\sigma_{t_J}),$$
for some subsets $J\subseteq I_1\cup I_2.$
\end{proposition}

\begin{proof}
Assume first of all that $I_1\cap I_2=\emptyset$. We have 
$$
0\le 2(n-1)-2\left(r_{I_1,\sigma_{t_1}}+r_{I_2,\sigma_{t_2}}\right),
$$
which gives

$$2r_{I_i,\sigma_{t_i}}\le 2n-(r_{I_1,\sigma_{t_1}}+1)-(r_{I_2,\sigma_{t_2}}+1),\ \forall i \in\{1,2\}.$$
Therefore, since $|I_i|\le   r_{I_i,\sigma_{t_i}}+1$, we 
have $2r_{I_i,\sigma_{t_i}}\le 2n-(|I_1|+|I_2|)$, $i=1,2$, hence we are in the hypotheses of Proposition \ref{intersection of joins}. This concludes the proof of the statement in the case when $I_1,I_2$ are disjoint. Indeed at each step of blow-up, the intersection of any two subvarieties that are being blown-up is a union of smaller subvarieties that have been previously blown-up.

\vskip.3cm

Assume now that $I_1\cap I_2=:I_{12}\ne \emptyset$. Set $I'_i:=I_i\setminus I_{12}$, for $i=1,2$. Notice that 
$$r_{I'_i,\sigma_{t_i}}=r_{I_i,\sigma_{t_i}}-|I_{12}|, \ \forall i\in\{1,2\}.$$
Moreover, 
\begin{equation}\label{join of join}
\J(\J(L_{I'_i},\sigma_{t_i}),L_{I_{12}})=\J(L_{I_i},\sigma_{t_i}), \ \forall i\in\{1,2\}.
\end{equation}
Also, 
$$\J(\J(L_{I'_1},\sigma_{t_1}),L_{I_{12}})\cap\J(\J(L_{I'_2},\sigma_{t_2}),L_{I_{12}})=
\J(\J(L_{I'_1},\sigma_{t_1})\cap\J(L_{I'_2},\sigma_{t_2}),L_{I_{12}}).$$
One proves the two last equalities of subvarieties using induction on $|I_{12}|$
 based on the case $|I_{12}|=0$ for which the statement is obvious.
With a similar argument as the one employed in the first case, we obtain that 
\begin{align*}
2r_{I'_i,\sigma_{t_i}}&\le 2n-(r_{I'_1,\sigma_{t_1}}+1)-(r_{I'_2,\sigma_{t_2}}+1)-4|I_{12}|,\\
& \le 2n-(|I'_1|+|I'_2|)),
\end{align*}
for $i=1,2$.
Therefore, 
\begin{equation}\label{intersection of joins extended}
\J(L_{I_1},\sigma_{t_1})\cap\J(L_{I_2},\sigma_{t_2})=
\J\left(\bigcup_J\J(L_J,\sigma_{t_J}),L_{I_{12}}\right)=\bigcup_J\J(L_{J\cup I_{12}},\sigma_{t_J}),
\end{equation}
where the union is taken over all subsets $J\subseteq I'_1\cup I'_2$ satisfying
$2|(I'_i\cup J)\setminus(I_i\cap J)|=|I'_1|+|I'_2|,\ \forall i\in\{1,2\},$
and for every such $J$, $t_J$ is the integer defined by the following equation
$2r_{J,\sigma_{t_J}}=2r_{I'_i, \sigma_{t_i}}-(|I'_1|+|I'_2|).$
The first equality of \eqref{intersection of joins extended} follows from Proposition
 \ref{intersection of joins}, the second inequality is an application of \eqref{join of join}.
\end{proof}

\begin{notation}\label{even odd notation}
For every $n$, let $\mathcal{T}:=\{(I,t):1\le r_{I,\sigma_t}\le n-2 \}$ be the set parametrizing all subvarieties  $\J(L_I, {\sigma}_{t})$ of $\PP^n$ of dimension between $1$ and $n-2$. 
\begin{enumerate} 
\item 
If $n$ is even, say $n=2\nu$, with $\nu\ge2$, 
we consider the following set
$$\mathcal{C}^{\textrm{even}}:=\{(I,t)\in\mathcal{T}: 2r_{I,\sigma_t}\le 2\nu-1 \textrm{ if } I\neq\emptyset\}.$$
In other terms, $(I,t)\in\mathcal{C}^{\textrm{even}}$ if $I=\emptyset$ and $1\le t<\nu$ or if  $I\neq\emptyset$ and $|I|+2t\le\nu$.
\item
If $n$ is odd, say 
$n=2\nu+1$, with $\nu\ge2$, we
consider the following set
$$\mathcal{C}^{\textrm{odd}}:=\{(I,t)\in\mathcal{T}: 2r_{I,\sigma_t}\le 2\nu \textrm{ if } I\neq\emptyset,\{1\}\}.$$
In other terms, $(I,t)\in\mathcal{C}^{\textrm{even}}$ if $I=\emptyset,\{1\}$ and $1\le t<\nu$ or if  $I\neq\emptyset,\{1\}$ and $|I|+2t\le\nu+1$.
\end{enumerate}

We consider $$\pi^\sigma:Y^\sigma\to X,$$   the iterated blow-up $Y^\sigma:=X^\sigma_{n+3,(n-2)}$ of $X:=X_{2\nu+3,(0)}$ (cf. Notation \ref{Ysigma}) along all subvarieties $\J(L_I,\sigma_t)$ such that $(I,t)\in\mathcal{C}^{\textrm{even}}$ 
if $n$ is even, and such that $(I,t)\in\mathcal{C}^{\textrm{odd}}$ if $n$ is odd, in increasing dimension.
We denote by $E_{I,\sigma_t}$ the corresponding  exceptional divisors.

Let us denote  $\Exc(\pi^\sigma)=\sum_{(I,t)\in\mathcal{C}^{\textrm{even}}}E_{I,\sigma_t}$
if $n$ is even and $\Exc(\pi^\sigma)=\sum_{(I,t)\in\mathcal{C}^{\textrm{odd}}}E_{I,\sigma_t}$ if $n$ is odd.
\end{notation}

\begin{proposition}\label{Ysigma smooth even}
For any $n$, the blown-up space $Y^\sigma$ is smooth. Moreover $\Exc(\pi^\sigma)$
is simple normal crossing.
\end{proposition}
\begin{proof}
At each level of the blow-up, the center is a disjoint union of smooth subvarieties.
Indeed, if $\J(L_{I_1},\sigma_{t_1})$ and $\J(L_{I_2},\sigma_{t_2})$ are two
subvarieties such that  $2r_{I_1,\sigma_{t_1}}=2r_{I_2,\sigma_{t_2}}\le n-1$,
 then we conclude by Proposition \ref{snc}. 
Otherwise, if $2r_{I,\sigma_{t}}\ge n$, the only subvarieties that we blow-up 
 are the strict transforms of the secant varieties $\sigma_t$'s, with 
$\frac{\nu+1}{2}\le t\le\nu$, in the case $n=2\nu$, and the strict transforms 
of the $\sigma_t$'s and of the pointed cones $\J(L_{\{1\}},\sigma_t)$, with $\frac{\nu}{2}+1\le t\le\nu$, if $n=2\nu+1$. It is easy to see that such pointed cones intersect along smaller joins of the same form, that have been previously blown-up.
\end{proof}


\subsubsection{Properties of the strict transform of $|\epsilon D|$}

Denote by $\tilde{D}^\sigma$ the strict transform of $D$, as in \eqref{definition D sigma}:
\begin{equation}\label{D tilde sigma}
\tilde{D}^{\sigma}:=dH-\sum_{i}m_i E_i-\sum_{r=1}^{n-2}\sum_{\substack{I,t:\\ r_{I,\sigma_t}=r}}k_{I,\sigma_t}E_{I,\sigma_t}.
\end{equation}

\begin{remark}
Notice that, as in the only linearly obstructed case (handled in Section \ref{section abundance n+2}), the integer $r$ in the above summation of exceptional divisors ranges up to $n-2$, as $D$
is movable (cfr. \cite[Theorem 5.3]{bdp3}).
\end{remark}

We use the sets $\mathcal{C}^{\textrm{even}}$ and $\mathcal{C}^{\textrm{odd}}$ introduced in Notation \ref{even odd notation} in the following lemma.

\begin{lemma}\label{tildeD intersect exc}
For any $n$, $s=n+3$ and $D$ as in Theorem \ref{abundance Q}, if 
 $k_{I,\sigma_t}>0$, then $2r_{I,\sigma_t}< n-1$. In particular such a pair 
$(I,t)\in\mathcal{T}$ belongs to $\mathcal{C}^{\textrm{even}}$ if $n$ is even, and to $\mathcal{C}^{\textrm{odd}}$, if $n$ is odd.
\end{lemma} 
 \begin{proof}
Let us  compute
\begin{align*}
\epsilon k_C& =\epsilon\max\left\{0, \left(\sum_{i=1}^{n-3}m_i-(n-3)d\right)
+(m_{n-2}+m_{n-1}-d)\right.+\\
& \ \  \ + (m_{n}+m_{n+1}-d)+(m_{n+2}+m_{n+3}-d)\Bigg\}\\ 
& \le-2(n-3)+3(n-2)
\\ & =n-3,
\end{align*}
where $k_C:=k_{\sigma_1}$ is the multiplicity of containment of the rational normal curve 
of degree $n$. If $k_C>0$, we also compute
\begin{align*}
\epsilon (k_C-d)&=\epsilon\left(\sum_{i=1}^{n-1}m_i-(n-1)d
+(m_{n}+m_{n+1}-d)+(m_{n+2}+m_{n+3}-d)\right)\\  
& \le-2(n-1)+2(n-2)\\ & =-4.
\end{align*}
Both of the above inequalities follow from  \eqref{abundance condition}, second line, and the inequality \eqref{extra nef condition}, computed
in the proof of Theorem  \ref{abundance Q}.
Hence, for every $t\ge0$, we obtain
\begin{align*}
\epsilon k_{\sigma_t}& =\epsilon\max\left\{0, k_C+(t-1)(k_C-d)\right\}\\
& \le \max\{0, n-3-4(t-1)\}.
\end{align*}
Furthermore, for every $I$ with $|I|\ge0$ and $t\ge0$, we obtain
\begin{align*}
\epsilon k_{I,\sigma_t}&=\epsilon\max\left\{0,k_{\sigma_t}+\sum_{i\in I}m_i-(|I|)d\right\}\\
& \le\max\{0,
(n-3-4(t-1))-2|I|\}\\ & =\max\{0,(n-|I|-2t+1)-(|I|+2t)\}.
\end{align*}

Hence, for every variety $\J(L_I,\sigma_t)\subset\PP^n$ 
with $r_{I,\sigma_t}=|I|+2t-1\le n$, we obtain
\begin{equation}\label{low mult of joins}
\epsilon k_{I,\sigma_t}\le \max\{0,n-1-2r_{I,\sigma_t}\}.
\end{equation}
\end{proof} 
 
\begin{corollary}\label{tildesigma transversal intersection}
In the same notation as Lemma \ref{tildeD intersect exc},
the divisor $\tilde{D}^\sigma$ intersects transversally each exceptional divisors.
\end{corollary}

\begin{theorem}\label{base point free tilde-sigma-D}
For any $n$, $s=n+3$ and $D$ as in Theorem \ref{abundance Q}. If the divisor $D$ 
is general, then $\tilde{D}^\sigma$ is base point free.
\end{theorem} 

We will show Theorem \ref{base point free tilde-sigma-D}  in Section
\ref{proof base point freeness D case even} (if $n$ is even) and in Section 
\ref{proof base point freeness D case odd} (is $n$ is odd).
The next corollary follows from  Bertini's Theorem.

\begin{corollary}\label{general tildesigma smooth}
In the notation of Theorem \ref{base point free tilde-sigma-D}, the general element of $|\tilde{D}^\sigma|$ is smooth.
\end{corollary}
 
\subsection{Proof of Theorem \ref{base point free tilde-sigma-D}, case $n$ even}
\label{proof base point freeness D case even}

Write $n=2\nu$, $\nu\ge2$.
Recall that the strict transform on $X:=X_{2\nu+3,(0)}$ of the secant variety 
$\sigma_\nu\subset\PP^{2\nu}$ is the fixed divisor
$$
\Sigma:=(\nu+1)H-\nu\sum_{i=1}^{2\nu+3}E_i.
$$

\begin{proposition}
In the above notation, the strict transform $\tilde{\Sigma}^\sigma$ of $\Sigma$ on $Y^\sigma$ is smooth. 
\end{proposition}
\begin{proof}
Recall that the singular locus of  $\sigma_\nu\subset\PP^{2\nu}$ is $\sigma_{\nu-1}\subset\sigma_\nu$ and, more precisely, the non-reduced union
of all joins $\J(L_I,\sigma_t)$ such that $k_{I,\sigma_t}(\Sigma)>1$. 
In particular we compute the following multiplicities. 
For every $I$ and $t$ such that $|I|,t\ge0$,  we have 
\begin{equation}\label{mult Sigma}
k_{I,\sigma_t}(\Sigma)=\max\{0,\nu-|I|-t+1\}.
\end{equation}
All subvarieties $\J(L_I,\sigma_t)$ such that $|I|+t\le \nu$ have been blown-up,
hence we conclude that $\pi^\sigma$ is a resolution of the singularities of $\Sigma$.
\end{proof}

Choose $\alpha\in\mathbb{N}$ such that
\begin{equation}\label{definition alpha}
\frac{k_C(D)}{\nu}
\le \alpha\le \min_{1\le i\le 2\nu+3}\left\{\frac{m_i}{\nu},d-m_i\right\}.
\end{equation}
\begin{lemma}
Under the assumptions of Theorem \ref{abundance Q}, such an integer $\alpha$ exists.
\end{lemma}
\begin{proof}
It is enough to prove that $\frac{k_C(D)}{\nu}\le \frac{m_i}{\nu}-1$ and that 
$\frac{k_C(D)}{\nu}\le d-m_i-1$ for all $i\in\{1\dots,2\nu+3\}$.
The first statement follows from the following computation: 
$$
\epsilon k_C(D)\le 2\nu-3\le 2\nu-1-\epsilon\nu\le \epsilon (m_i-\nu).
$$
The first inequality follows from \eqref{low mult of joins}. The second inequality follows from
the assumption 
 $\epsilon\ll1$; in fact it is enough to take $\epsilon\le\frac{4}{n}=\frac{2}{\nu}$. The last inequality follows from 
\eqref{abundance condition}. Similarly, one proves the second statement by observing that
$$
\epsilon k_C(D)\le 2\nu-3\le 2\nu-\epsilon\nu\le \epsilon \nu(d-m_i-1).
$$
The last inequality follows from \eqref{extra nef condition}.
\end{proof}

Consider the linear system 
$$
|D'|:=|D-\alpha \Sigma|=|d'H-\sum_{i=1}^{2\nu+3}m'_iE_i|,
$$
with $d':=d-\alpha(\nu+1)$ and $m'_i:=m_i-\alpha\nu$, for all $i\in\{1\dots,2\nu+3\}$.
We have
$$
|D'|+\alpha\Sigma\subseteq|D|. 
$$

\begin{proposition}\label{Delta' smooth even} 
In the above notation, the  linear system $|D'|$ is non-empty.
Moreover $|D'|$ has only linear base locus and  
 the strict transform in $Y^\sigma$ of $|D'|$  is base point free.
\end{proposition}
\begin{proof}
For every $i\in\{1\dots,2\nu+3\}$, 
since by \eqref{definition alpha} we have 
$\alpha\le \frac{m_i}{\nu}$ and $\alpha\le d-m_i$,
then $m'_i\ge0$ and $d\ge m'_i$ respectively.
Moreover, let us compute 
\begin{align*}
k_C(D')&=\max\left\{0,\sum_{i=1}^{2\nu+3}(m_i-\alpha\nu)-2\nu(d-\alpha(\nu+1))\right\}\\
& =\max\left\{0,k_C(D)-\alpha\nu\right\}\\
&=0.
\end{align*}
The last equality follows from \eqref{definition alpha}: $\alpha\ge \frac{k_C(D)}{\nu}$.
This proves the first statement, namely that $\Delta'$ is effective, see \cite[Theorem 5.1]{bdp3}.

To prove the second statement,
notice that if $|I|\le\nu$, then the pair $(I,0)\in\mathcal{C}^{\textrm{even}}$, hence 
the corresponding linear subspace $L_I\subset\PP^{2\nu}$ has been blown-up.
Otherwise, if $|I|\ge\nu+1$, we claim that $k_I(D')=0$.  Theorem \ref{gg theorem Xr}
implies the second statement, namely that $\tilde{D}^\sigma$ is a globally generated divisor.
To prove the claim,  for every $I$ we choose $J\subset I$ with $|J|=\nu$ and
 we compute
\begin{align*}
\epsilon k_{I}(D)&=\max\left\{0, \epsilon\left(\sum_{i\in I}m_i-(|I|-1)d\right)+
\epsilon\alpha(|I|-\nu-1)\right\}\\
&=\max\left\{0,\epsilon\left(\sum_{i\in J}m_i-(\nu-1)d\right)+
\epsilon\left(\sum_{i\in I\setminus J}(m_i-d)\right)+\epsilon \alpha(|I|-\nu-1)\right\}\\
&\le \max\{0,-\epsilon \alpha\}\\ &=0.
\end{align*}
The inequality holds because 
$\epsilon \left(\sum_{i\in J}m_i-(\nu-1)d\right)\le0$ by \eqref{low mult of joins} and
$\alpha(|I|-\nu)\le \sum_{i\in I\setminus J}(d-m_i)$ by \eqref{definition alpha}.
\end{proof}

\begin{proposition}\label{linear series of strict transform even}
In the above notation, there exist non-negative numbers
 $\alpha_{I,\sigma_t}\in\mathbb{Z}$ such that 
$$
|\tilde{D'}^\sigma|+\alpha\tilde{\Sigma}^\sigma+\sum_{(I,t)\in\mathcal{C}^{\textrm{even}}}
\alpha_{I,\sigma_t} E_{I,\sigma_t}\subseteq 
|\tilde{D}^\sigma|.
$$
\end{proposition}
\begin{proof} 
In order to prove the statement, we compare  the strict transform, 
$|\tilde{D'}^\sigma|+\alpha\tilde{\Sigma}^\sigma$  with that of $D$.

Recall that, for any $1\le t<\nu$ and $I$ such that $|I|\ge 0$, we have $k_{I,\sigma_t}(D')=0$ 
by Proposition \ref{Delta' smooth even}. Moreover for all $I$ and $t$,
  \eqref{low mult of joins} implies that 
\begin{equation}\label{Sigma is more singular than Delta}
k_{I,\sigma_t}(\alpha \Sigma)=\alpha k_{I,\sigma_t}(\Sigma)\ge k_{I,\sigma_t}(D).
\end{equation}

This, together with \eqref{mult Sigma} and the computation of $k_I(D)$ made in the
proof of Proposition \ref{Delta' smooth even}, implies the following 
linear equivalence of divisors
$$\tilde{D}^\sigma\sim \tilde{D'}^\sigma+\alpha\tilde{\Sigma}^\sigma+
\sum_{(I,t)\in\mathcal{C}^{\textrm{even}}}
\alpha_{I,\sigma_t}E_{I,\sigma_t},$$ 
where 
 $\alpha_{I}:=k_{I}(\alpha\Sigma)+k_{I}(D')
-k_{I}(D)$ 
if  $1\le |I|\le \nu$, while 
 $\alpha_{I,\sigma_t}:=k_{I,\sigma_t}(\alpha\Sigma)-k_{I,\sigma_t}(D)$ for 
$1\le t\le\nu$, $\nu+1\le |I|\le 2\nu$.
Finally, from \eqref{Sigma is more singular than Delta} we obtain $a_{I,\sigma_t}\ge0$,
for all $I$ and $t$.
\end{proof}

\begin{proof}[Proof of Theorem \ref{base point free tilde-sigma-D}, case $n$ even]
We recall that a general member of $|D|$ vanishes at every point of $J(L_I,\sigma_t)$ 
with multiplicity equal to $k_{I,\sigma_t}$ and that $|\tilde{D}^\sigma|$ is the
 strict transform of $D$ under the blow-up of its secant base locus. Assume that the base locus, $\Bs|\tilde{D}^\sigma|$, is non-empty.
By Proposition \ref{linear series of strict transform even}, we have
$$
\Bs|\tilde{D}^\sigma| \subset \Bs|\tilde{D'}^\sigma|\cup \tilde{\Sigma}\cup 
\bigcup_{(I,t)\in\mathcal{C}^{\textrm{even}}}E_{I,\sigma_t}.
$$

Note that the divisor $|\tilde{D'}^\sigma|$ is base point free by Theorem \ref{gg theorem Xr}, therefore  $\Bs|\tilde{D'}^\sigma|$ is empty. Hence
$$
\Bs|\tilde{D}^\sigma| \subset  \tilde{\Sigma}\cup 
\bigcup_{(I,t)\in\mathcal{C}^{\textrm{even}}}E_{I,\sigma_t}.
$$

 Assume there is a base point for $|\tilde{D}^\sigma|$ on one of the exceptional divisors $E_{I,\sigma_t}$, $(I,t)\in\mathcal{C}^{\textrm{even}}$, or on $\tilde{\Sigma}$.
This implies the existence of a point in a join 
$J(L_I,\sigma_t)$, or in $\Sigma$, on which the general divisor $D$ has infinitesimal tangencies. The proof of Proposition \ref{infinitesimal base locus} implies that no point of $J(L_I,\sigma_t)$, 
nor of $\Sigma$, carries any infinitesimal information. This leads to a contradiction.

\end{proof}

\subsection{Proof of Theorem \ref{base point free tilde-sigma-D}, case $n$ odd}
\label{proof base point freeness D case odd}

Recall that the strict transform on $X:=X_{2\nu+3,(0)}$ of the cone $\J(L_{\{1\}},\sigma_\nu)\subset\PP^{2\nu}$ is the divisor
$$
\Gamma:=(\nu+1)H-(\nu+1)E_1-\nu\sum_{i=2}^{2\nu+4}E_i.
$$

\begin{proposition}
In the above notation, the strict transform $\tilde{\Gamma}^\sigma$ of $\Gamma$ on $Y^\sigma$ is smooth. 
\end{proposition}
\begin{proof}
Recall that the singular locus of  $\J(L_{\{1\}},\sigma_\nu)\subset\PP^{2\nu}$ is $\J(L_{\{1\}},\sigma_{\nu-1})\subset \J(L_{\{1\}},\sigma_\nu)$ and, more precisely, the non-reduced union
of all joins $\J(L_I,\sigma_t)$ such that $k_{I,\sigma_t}(\Sigma)>1$. 
In particular we compute the following multiplicities. 
For every $I$ and $t$ such that $|I|,t\ge0$,  we have 
\begin{equation}\label{mult Sigma odd}
k_{I,\sigma_t}(\Sigma)=\max\left\{0, \nu-|I|-t+1+\delta_I\right\},
\end{equation}
where $\delta_I$ is defined as
\begin{equation}\label{Kronecker}
\delta_I:=\delta_{1,I}=\left\{\begin{array}{ll}
1 & \textrm{if } 1\in I,\\
0 & \textrm{if } 1\notin I.\\
\end{array}\right.
\end{equation}
All subvarieties $\J(L_I,\sigma_t)$ such that $|I|+t\le \nu+\delta_I$ have been blown-up,
hence we conclude that $\pi^\sigma$ is a resolution of the singularities of $\Sigma$.
\end{proof}

Choose $\beta\in\mathbb{N}$ such that
\begin{equation}\label{definition beta}
\frac{k_C(D)}{\nu}
\le \beta\le \min_{1\le i\le 2\nu+4}\left\{\frac{m_1}{\nu+\delta_{\{i\}}},d-m_i\right\},
\end{equation}
where $\delta_{\{i\}}$ is the  Kronecker delta defined in \eqref{Kronecker}.
\begin{lemma}
Under the assumptions of Theorem \ref{abundance Q}, such an integer $\beta$ exists.
\end{lemma}
\begin{proof}
For every $i\in\{1,\dots, 2\nu+4\}$, we compute
$$
\epsilon k_C(D)\le 2\nu-2\le 2\frac{\nu^2}{\nu+\delta_{\{i\}}}-\epsilon\nu\le 
\epsilon\nu\left(\frac{m_1}{\nu+\delta_{\{i\}}}-1\right).
$$
In the above expression, 
the first inequality follows from \eqref{low mult of joins}. The second inequality follows from
the assumption    $\epsilon\ll1$; in fact it is enough to take 
 $\epsilon\le\frac{8}{n^2-1}=\frac{2}{\nu^2+\nu}$.
The last inequality follows from 
\eqref{abundance condition}. Furthermore we have
$$
\epsilon k_C(D)\le 2\nu-2\le 2\nu-\epsilon\nu\le \epsilon \nu(d-m_i-1).
$$
The last inequality follows from \eqref{extra nef condition}.
\end{proof}

Consider the linear system 
$$
|D'|:=|D-\beta \Gamma|=|d'H-\sum_{i=1}^{2\nu+4}m'_iE_i|,
$$
with $d':=d-\beta(\nu+1)$,   $m'_1:=m_1-\beta(\nu+1)$ and $m'_i:=m_i-\beta\nu$,
 for all $i\in\{2\dots,2\nu+4\}$.
We have
$$
|D'|+\beta\Gamma\subseteq|D|. 
$$

\begin{proposition}\label{Delta' smooth odd}
In the above notation, the  linear system $|D'|$ is non-empty.
Moreover $|D'|$ has only linear base locus and  
 the strict transform in $Y^\sigma$ of $|D'|$  is base point free.
\end{proposition}
\begin{proof}
By \eqref{definition beta}, we have $0\le m'_i\le d$ for all $i\in\{1,\dots, 2\nu+4\}$.
Moreover, as in the proof of Proposition \ref{Delta' smooth even}, we compute 
$$
k_C(D')=\max\left\{0, \epsilon\left(\sum_{i=1}^{2\nu+4}m'_i-(2\nu+1)d'\right)
\right\}=\max\left\{0,k_C(D)-\beta\nu\right\}=0.
$$
This proves the first statement.

To prove the second statement, 
we notice that if $|I|\le\nu+1$ the pair $(I,0)\in\mathcal{C}^{\textrm{odd}}$, while 
 we claim that $k_I(D')=0$ if $|I|\ge\nu+2$. This and Theorem \ref{gg theorem Xr}
imply the second statement.
We prove the claim for $1\in I$; the case $1\notin I$ is similar and we leave the details to the reader.
 Choose $J\subset I$ with $|J|=\nu+1$ and $1\in J$. Let us 
 compute

\begin{align*}
\epsilon k_{I}(D)&=\max\left\{0, \epsilon\left(\sum_{i\in I}m_i-(|I|-1)d\right)+
\epsilon\beta(|I|-\nu-2)\right\}\\
&=\max\left\{0,\epsilon\left(\sum_{i\in J}m_i-(\nu)d\right)+
\epsilon\left(\sum_{i\in i\setminus J}(m_i-d)\right)+\epsilon \beta(|I|-\nu-2)\right\}\\
&\le \max\{0,-\epsilon \beta\}\\ &=0.
\end{align*}
The inequality holds thanks to \eqref{low mult of joins} and
 \eqref{definition beta}.
\end{proof}

\begin{proposition}\label{linear series of strict transform odd}
In the above notation, there exist non-negative numbers
 $\beta_{I,\sigma_t}\in\mathbb{Z}$ such that 
$$
|\tilde{D'}^\sigma|+\beta\tilde{\Gamma}^\sigma+\sum_{(I,t)\in\mathcal{C}^{\textrm{odd}}}
\beta_{I,\sigma_t} E_{I,\sigma_t}\subseteq 
|\tilde{D}^\sigma|.
$$
\end{proposition}
\begin{proof}
The proof follows the same lines as that of Proposition 
 \ref{linear series of strict transform even} and it uses Proposition \ref{Delta' smooth odd}. We leave the details to the reader. 
\end{proof}

\begin{proof}[Proof of Theorem \ref{base point free tilde-sigma-D}, case $n$ odd]
The proof follows the same idea as that of the case $n$ even, at the end of Section
\ref{proof base point freeness D case even}.

\end{proof}

\subsubsection{The pair $(X,\epsilon D)$ is lc, for $D$ general}

Notice that  the canonical divisor of $Y^\sigma=X^\sigma_{n+3,(n-2)}$ is
$$
K_{Y^\sigma}=-(n+1)H+(n-1)\sum E_i+\sum_{r=1}^{n-2}(n-r-1)
\sum_{\substack{I,t:\\ r_{I,\sigma_t}=r}}E_{I,\sigma_t}.
$$

We are now ready to prove Theorem \ref{abundance Q} for $s=n+3$.

\begin{proof}[Proof of Theorem \ref{abundance Q}, case $s=n+3$]
By Corollary \ref{general tildesigma smooth},
$(Y^\sigma,\tilde{\Delta})$ 
is log smooth and $\pi:Y\to X$ is a log resolution of $(X,\Delta)$.

To complete the proof, 
similarly to the case of only linearly obstructed divisors, 
we are going to show that \eqref{abundance condition} implies 
 
\begin{equation}\label{lc' condition}
  \epsilon k_{I,\sigma_t}  \le n-|I|-2t+1, 
\quad  \forall I(r), \ 2\le r_{I,\sigma_t}\le n-2,
\end{equation}
that in turns implies that $\textrm{discrep}(X,\Delta)\ge-1$.
This follows from the inequalities \eqref{low mult of joins} computed in the proof of  Proposition \ref{tildeD intersect exc}.
\end{proof}



\begin{thebibliography}{99}



\bibitem{AC}
C. Araujo, C. Casagrande
{\it On the Fano variety of linear spaces contained in two odd-dimensional quadrics},
Geom. Topol. 21 (2017), no. 5, 3009--3045.

\bibitem{am}
C. Araujo, A. Massarenti,
{\it Explicit log Fano structures on blow-ups of projective spaces},
 Proc. Lond. Math. Soc. (3) 113 (2016), no. 4, 445--473

\bibitem{bau}
S.~Bauer, Parabolic bundles, elliptic surfaces and SU(2)-representation spaces of
genus zero Fuchsian groups, Math. Ann. 290 (1991), 509��--526.

\bibitem{BCHM}
C. Birkar, P. Cascini, C. Hacon, J. McKernan: {\it Existence of minimal models for varieties of log general type},
J. Amer. Math. Soc. 23, 405--468, (2010)

\bibitem{bdp1}
 M.C.~Brambilla, O.~Dumitrescu and E.~Postinghel, {\it On a notion of speciality of linear systems in $\mathbb{P}^n$},
Trans. Am. Math. Soc. no. 8, 5447--5473 (2015).

\bibitem{bdp3} M. C. Brambilla, O. Dumitrescu and E. Postinghel,
{\it On the effective cone of $\mathbb{P}^n$ blown-up at $n+3$ points},
 Exp. Math. 25, no. 4, 452--465 (2016)

\bibitem{CDDGP} S.~Cacciola, M.~Donten-Bury, O.~Dumitrescu, A.~Lo Giudice, J.~Park, {\it Cones of divisors of blow-ups of projective spaces}, Matematiche (Catania) 66 (2011), no. 2, 153--187.

\bibitem{CA} A.M.~Castravet, {\it The Cox ring of $\overline{\mathcal{M}}_{0,6}$}, Trans. Am. Math. Soc. no. 7, Vol. 361 (2009), 3851--3878.
\bibitem{CT} A.M.~Castravet and J.~Tevelev, {\it Hilbert's 14th problem and Cox rings}, Compos. Math. 142 (2006), no. 6, 1479--1498.
\bibitem{CT2} A.M.~Castravet and J.~Tevelev, {\it $\overline{\mathcal{M}}_{0,n}$ is not a Mori Dream Space}, Duke Math. J., 164, no. 8 (2015), 1641--1667 

\bibitem{Ciliberto} C.~Ciliberto, {\it Geometrical aspects of polynomial
  interpolation in more variables and
of Waring's problem}, European Congress of Mathematics, Vol. I 
(Barcelona, 2000), 289--316, Progr. Math., 201, Birkh\"auser, Basel, 2001. 

\bibitem{dp}
O.~Dumitrescu and E.~Postinghel, {\it Vanishing theorems for linearly obstructed divisors}, J. Algebra, 477, 312--359 (2017)

\bibitem{dp-pos1}
O.~Dumitrescu and E.~Postinghel, {\it Positivity of divisors on blown-up projective spaces, I}, preprint arXiv:1506.04726 

\bibitem{Fk}
S. Fukuda, 
{\it On numerically effective log canonical divisors}, 
Int. J. Math. Math. Sci. 30, no. 9, 521--531, (2002)

\bibitem{Froberg} R.~Fr\"oberg,
{\it An inequality for Hilbert series of graded algebras},
Math. Scand. 56 , no. 2, 117--144, (1985)

\bibitem{GKM}
A. Gibney, S.  Keel and I.  Morrison,
{\it Towards the ample cone of {$\overline M_{g,n}$}},
 J. Amer. Math. Soc. 15, 273--294, (2002)
	
\bibitem{gon}
Y. Gongyo, 
{\it Remarks on the non-vanishing conjecture}, 
Algebraic geometry in east Asia-Taipei 2011, 107-–116, 
Adv. Stud. Pure Math., 65, Math. Soc. Japan, Tokyo (2015).

\bibitem{karu} J.~Gonzalez, K.~Karu,
{\it Some non-finitely generated Cox rings},
Compos. Math. 152 (2016), no. 5, 984--996. 

\bibitem{HMX}
C. D. Hacon, J. McKernan,C.  Xu,
{\it ACC for log canonical thresholds}, 
Ann. of Math. (2) 180 (2014), no. 2, 523-€"-571


\bibitem{HuKeel}
Y. Hu and S. Keel,
{\it Mori Dream Spaces and GIT},
Michigan Math. J. 48, 331--348 (2000)

\bibitem{Iarrobino}
A.~Iarrobino,
{\it Inverse system of symbolic power III. Thin algebras and fat points}, 
Compositio Math. 108 (1997), no. 3, 319--356. 

\bibitem{Ka}
M.~Kapranov, {\it Veronese curves and Grothendieck-Knudsen moduli space $\overline{M_{0,n}}$}, J.
Algebraic Geom. 2 (1993), no. 2, pp. 239--262.

\bibitem{KMM}
S.~Keel, K.~Matsuki and J.~McKernan,  
{\it Log abundance theorem for threefolds},
Duke Math. J. 75, no. 1, 99--119, (1994)

\bibitem{KeMc} S.~Keel and J.~McKernan, {\it Contraction of extremal rays on $\overline{M_{0,n}}$}, 
arXiv:alg-geom/9607009, (1996).

\bibitem{KM}
 J.~Koll\'ar and S.~Mori, 
{\it Birational geometry of algebraic varieties}, 
Cambridge Tracts in Mathematics, vol. 134, Cambridge University Press, Cambridge, (1998),
 With the collaboration of C.~H.~Clemens and A.~Corti, Translated from the 1998 Japanese
original.

\bibitem{LKH}
J. Hausen, S. Keicher, A. Laface
\emph{On blowing-up the weighted projective plane},  
Math. Z. 290 (2018), no. 3-4, 1339--1358. 

\bibitem{LosevManin} 
A. Losev, Y. Manin, 
{\it New moduli spaces of pointed curves and pencils of flat connections},
Michigan Math. J. Volume 48, Issue 1, 2000, 443--472.

\bibitem{Moon}
H.B. Moon, S.B. Yoo, {\it Birational Geometry of the Moduli Space of Rank 2 Parabolic Vector Bundles on a Rational Curve}, 
Int. Math. Res. Not. IMRN 2016, no. 3, 827--859. 

\bibitem{Mukai}
S. Mukai, {\it Counterexample to Hilbert's fourteenth problem for the
 $3$-dimensional additive
group}, RIMS preprint \#1343, Kyoto, (2001)

\bibitem{Mu}
S. Mukai, {\it Finite generation of the Nagata invariant rings in A-D-E
cases}, 2005, RIMS Preprint \# 1502.

\bibitem{veronese}
G. Veronese, {\it
Behandlung der projectivischen Verh\"altnisse der R\"aume von 
verschiedenen Dimensionen durch das Princip des Prjjicirens und Schneidens}, (German), 
Math. Ann. 19 (1881), no. 2, 161--234.



\end{thebibliography}
\end{document}